\documentclass[12pt]{article}

\usepackage{amssymb}
\usepackage{amsthm}
\usepackage{amsmath}
\usepackage{graphicx}
\usepackage{fullpage}
\usepackage{color}
\usepackage{enumerate}
 \numberwithin{equation}{section}
\usepackage{booktabs}
\allowdisplaybreaks
\newtheorem{innercustomthm}{Theorem}
\newenvironment{customthm}[1]
  {\renewcommand\theinnercustomthm{#1}\innercustomthm}
  {\endinnercustomthm}
\usepackage[pdftex]{hyperref}
 \usepackage{mathtools}
 \mathtoolsset{showonlyrefs}

\theoremstyle{plain}
\newtheorem{thm}{Theorem}[section]
\newtheorem{cor}[thm]{Corollary}

\newtheorem{lem}[thm]{Lemma}
\newtheorem{prop}[thm]{Proposition}

\newtheorem*{thm*}{Theorem}

\theoremstyle{definition}
\newtheorem{defn}[thm]{Definition}

\theoremstyle{remark}

\newtheorem{rem}[thm]{Remark}

\newcommand{\M}{\mathbb{M}}
\newcommand{\N}{\mathbb{N}}
\newcommand{\R}{\mathbb{R}}

\newcommand{\ffint}{\iint_{Q_r} \!\!\!\!\!\!\!\!\!\!\!\!\!\!\text{-----}}
\newcommand{\ffintXT}{\iint_{Q_r(x,t)} \!\!\!\!\!\!\!\!\!\!\!\!\!\!\!\!\!\!\!\!\!\! \text{-----}}

\newcommand{\fffint}{\iint_{Q_1} \!\!\!\!\!\!\!\!\!\!\!\!\!\!\text{-----}}
\newcommand{\fthetarint}{\iint_{Q_{\vartheta r}} \!\!\!\!\!\!\!\!\!\!\!\!\!\!\!\!\text{-----}}
\newcommand{\fthetaint}{\iint_{Q_{\vartheta}} \!\!\!\!\!\!\!\!\!\!\!\!\!\!\text{-----}}

\newcommand{\Div}{\text{div}}
\newcommand{\bp}{\begin{proof}[\ensuremath{\mathbf{Proof}}]}
\newcommand{\ep}{\end{proof}}
\newcommand{\bv}{{\bf v}}

\newcommand{\bu}{{\bf u}}
\newcommand{\bg}{{\bf g}}

\newcommand{\bw}{{\bf w}}

\newcommand{\Mmn}{\M^{m\times n}}

\newcommand{\be}{\begin{equation}}
\newcommand{\ee}{\end{equation}}

\begin{document}

% Title 
\title{Partial regularity for type two doubly nonlinear parabolic systems}

\author{Ryan Hynd\footnote{Department of Mathematics, MIT.  Partially supported by NSF grant DMS-1554130 and an MLK visiting professorship.}}

\maketitle 

%  Abstract  
\begin{abstract}
We consider weak solutions $\bv:U\times (0,T)\rightarrow \R^m$ of the nonlinear parabolic system
$$
D\psi(\bv_t)=\Div DF(D\bv),
$$
where $\psi$ and $F$ are convex functions. This is a prototype for more general doubly nonlinear evolutions which arise in the study of structural properties of materials. Under the assumption that the second derivatives of $F$ are H\"older continuous, we show that $D^2\bv$ and $\bv_t$ are locally H\"older continuous except for possibly on a lower dimensional subset of $U\times (0,T)$.  Our approach relies on two integral identities, decay of the local space-time energy of solutions, and fractional time derivative estimates for $D^2\bv$ and $\bv_t$. 
\end{abstract}

%%%%%%%%%%%%%%%%%%%%%%%%%%%%%%%%%%%%%%%%%%%%%%%%%%%%%%%%%
\section{Introduction}
% Introductory comments 
A doubly nonlinear evolution is a flow that typically involves a nonlinearity in the time derivative of a particular quantity of interest.  Such flows arise in the study of phase transitions \cite{Blanchard, Bonfanti,ColliLut, Visintin}, models for fracture and crack fronts \cite{FraMie, Larsen, MieCrack}, and hysteresis effects in materials \cite{MieRos, Visintin2}.  In the very simplest modeling scenarios, the flows in question consist of systems of nonlinear PDE.  To date, there have been 
many important results on the existence \cite{AGS, Arai,  Colli2, Colli, Mielke} and on the large time behavior of solutions to these systems \cite{HynTAM, HynLin, ManVes, Schimp}.  On the other hand, there are very few results involving the regularity or smoothness properties of such solutions. This is the topic of this paper.

% System  
\par In what follows, we will consider solutions $\bv: U\times(0,T)\rightarrow \R^m$ of the system of PDE 
\be\label{mainPDE}
D\psi(\bv_t)=\Div DF(D\bv)
\ee
where $U\subset\R^n$ is a bounded domain with smooth boundary and $T>0$. Here $\psi:\R^m\rightarrow\R$ and $F:\M^{m\times n}\rightarrow \R$ are convex functions and $\M^{m\times n}$ is the space of $m\times n$ matrices with real entries. We may write $\bv=(v^1,\dots, v^m)$ in terms of its $m$ component functions $v^i=v^i(x,t)$; this allows us to conveniently express the time derivative $\bv_t=(v^i_t)\in \R^m$ and spatial gradient matrix of
$D\bv=(v^i_{x_j})\in \M^{m\times n}$ of $\bv$.

\par Considering $\psi$ as a function of $w=(w^i)\in \R^m$ and $F$ as a function of $M=(M^i_j)\in \Mmn$ also allows us to write
$$
D\psi=(\psi_{w^1}, \dots,\psi_{w^m})\in \R^m\quad \text{and}\quad 
DF
=\left(\begin{array}{ccc}
F_{M^1_1} & \dots  & F_{M^1_n}\\
& \ddots &  \\
F_{M^m_1}& \dots  & F_{M^m_n}\\
\end{array}\right)\in \Mmn.
$$
With this notation, the system \eqref{mainPDE} can be expressed as the system of $m$ equations
$$
\psi_{w_i}(\bv_t)=\sum^n_{j=1}\left(F_{M^i_j}(D\bv)\right)_{x_j}=\sum^m_{k=1}\sum^n_{j,\ell=1}F_{M^i_jM^k_\ell}(D\bv) v^k_{x_\ell x_j}
$$
$i=1, \dots, m$.  In particular, as $D\psi$ is in general nonlinear, we may interpret the system \eqref{mainPDE} as a type of fully nonlinear system of parabolic PDE.

% Assumptions on psi and F
\par Unless otherwise noted, we will always suppose
$$
\psi\in C^2(\R^m)\quad \text{and}\quad F\in C^2(\Mmn).
$$
Another standing assumption will be that there are $\theta,\lambda, \Theta,\Lambda>0$ for which
\begin{equation}\label{UnifConv}
\theta |w_1-w_2|^2\le \left(D\psi(w_1)-D\psi(w_2)\right)\cdot\left(w_1-w_2\right)\le 
\Theta |w_1-w_2|^2\end{equation}
for each  $w_1,w_2\in \R^m$ and
\begin{equation}\label{UnifConv2}
\lambda |M_1-M_2|^2\le \left(DF(M_1)-DF(M_2)\right)\cdot\left(M_1-M_2\right)\le \Lambda |M_1-M_2|^2
\end{equation}
for each $M_1,M_2\in \Mmn$.  Here $M\cdot N:=\text{tr}(M^tN)$ and $|M|:=(M\cdot M)^{1/2}$ for each $M,N\in\Mmn$. In particular, $\psi$ and $F$ will always assumed to be uniformly convex and to grow quadratically.

% Main Theorem 
\par For now, we will postpone providing definitions of a weak solution (Definition \ref{WeakSolnDefn}), the Hessian of mapping (equation \eqref{HessianOfbu}) and parabolic Hausdorff measure (Definition \ref{ParaHausMeas}) until later in this work. We only emphasize here that a weak solution is a solution for which \eqref{mainPDE} holds in an integral sense and has integrability properties as determined by natural identities satisfied by any smooth solution of \eqref{mainPDE}. The main assertion of this paper is as follows and contends that that every weak solution $\bv$ of \eqref{mainPDE} is a classical solution except for possibly on a lower dimensional subset of $U\times (0,T)$.

\begin{customthm}{1}\label{mainThm}
Assume $\bv$ is a weak solution of \eqref{mainPDE} in $U\times (0,T)$ and define 
$$
{\cal O}:=\{(x,t)\in U\times(0,T): D^2\bv \;\text{and}\; \bv_t \;\text{are H\"older continuous in a neighborhood of $(x,t)$}\}.
$$
Further suppose that each of the second derivatives of $F$ are H\"older continuous. Then 
there is a $\beta\in (0,1]$ such that 
$$
{\cal P}^{n+2-2\beta}(U\times (0,T)\setminus{\cal O})=0.
$$ 
\end{customthm}

% Previous results 
\par It is reasonable to wonder if the conclusion of Theorem \ref{mainThm} is sharp. While we do not offer a precise estimate on the parabolic Hausdorff dimension of $U\times (0,T)\setminus{\cal O}$, 
we do know that weak solutions can have singularities. Even in the stationary case, weak solutions $\bu:U\rightarrow \R^m$ of
$$
-\Div(DF(D\bu))=0
$$
will not in general be $C^1$ at every point in $U$ \cite{DeG, Guisti, Lawson}.    Theorem \ref{mainThm} has also been
previously verified for the gradient flow system
$$
\bv_t=\Div DF(D\bv),
$$
which corresponds to \eqref{mainPDE} when $\psi(w)=\frac{1}{2}|w|^2$ \cite{Campanato1}. We remark that this result has recently been refined \cite{DuzMin, DuzMinSte}, where it was established that ${\cal P}^{n-\delta}(U\times (0,T)\setminus {\cal O})=0$ for some $\delta>0$.  

\par In a previous study \cite{HynTAM}, we analyzed the system 
\be\label{Fquadratic}
D\psi(\bv_t)=\Delta\bv.
\ee
This corresponds to the particular case of \eqref{mainPDE} with $F(M)=\frac{1}{2}|M|^2$.  We showed that if a weak solution $\bv$ satisfies $$\int^T_0\int_U|D\bv_t|^2dxdt<\infty,$$ 
then $D^2\bv$ and $\bv_t$ are H\"older continuous in a neighborhood of almost every point in $U\times (0,T)$. 

\par In this paper, we will incorporate the integrability $D\bv_t\in L^2_{\text{loc}}(U\times(0,T);\Mmn)$ into the definition of weak solution and also show that we can always construct such a weak solution (Appendix \ref{DirichletAppendix}).  Then we will improve our previous regularity result for \eqref{Fquadratic} by obtaining a local regularity condition for solutions of the system \eqref{mainPDE} and then showing that this local regularity condition holds on a lower dimensional set as measured by parabolic Hausdorff measure.  The keys to our enhanced insight are a local energy decay property described in Lemma \ref{BlowUplemma} and fractional time derivative estimates for $\bv_t$ \eqref{FractionVtDeriv} and $D^2\bv$ \eqref{FractionD2VDeriv}; these results largely rely on the assumption that the second derivatives of $F$ are H\"older continuous.

\par Let us also briefly remark on the case of the scalar equation
\be\label{ScalarDNE}
\psi'(v_t)=\Div DF(Dv),
\ee
which corresponds to \eqref{mainPDE} when $m=1$. Here $v:U\times(0,T)\rightarrow \R$, $\psi\in C^2(\R)$ and $F\in C^2(\R^n)$. Using the Legendre transform of $\psi$, we may write \eqref{ScalarDNE} as the fully nonlinear parabolic equation
$$
v_t=(\psi^*)'\left(D^2F(Dv)\cdot D^2v\right).
$$
We suspect that if $D^2F$ is H\"older continuous, then $D^2v$ and $v_t$ are (everywhere) H\"older continuous. We plan to investigate this possibility in a future study.

% Type one/Type two 
\par Lastly, we remark that equation \eqref{mainPDE} is known as a doubly nonlinear parabolic system of the {\it second} type.  A doubly nonlinear parabolic system of the {\it first} type is of the form
\begin{equation}\label{FirstTypePDE}
\partial_t\left(D\psi\left(\bv\right)\right)=\Div DF(D\bv).
\end{equation} 
This terminology likely originated in the monograph \cite{Visintin}. The essential difference between \eqref{mainPDE} and \eqref{FirstTypePDE} is that \eqref{mainPDE} is fully nonlinear while \eqref{FirstTypePDE} is quasilinear.  Nevertheless, solutions of both systems exhibit partial regularity. We recently showed that for a weak solution $\bv$ of \eqref{FirstTypePDE}, $D\bv$ is locally H\"older continuous except for possibly at points confined to a lower dimensional subset of $U\times (0,T)$ \cite{Hynd}.

%%%%%%%%%%%%%%%%%%%%%%%%%%%%%%%%%%%%%%%%%%%%%%%%%%%%%%%%%
\section{Two identities}
Our first goal is to derive two integral identities.  The first identity will be obtained by multiplying both sides of \eqref{mainPDE} by $\phi \bv_t$ and integrating by parts; here $\phi\in C^\infty_c(U\times(0,T))$. Likewise, the second identity is essentially derived by multiplying \eqref{mainPDE} by $\phi \bv_{tt}$ and then integrating by parts. We emphasize that in this section we will only consider classical solutions, and in the following section, we will prove these identities for appropriately defined weak solutions. 

%First identity 
\begin{prop}
Assume $\phi\in C^\infty_c(U\times(0,T))$ and $\bv\in C^\infty(U\times(0,T);\R^m)$ is a solution of \eqref{mainPDE}. Then 
\begin{equation}\label{MainIdentity}
\frac{d}{dt}\int_U \phi F(D\bv)dx+\int_U\phi D\psi(\bv_t)\cdot \bv_t dx=\int_U\left(F(D\bv)\phi_t -\bv_t\cdot DF(D\bv)D\phi \right)dx
\end{equation}
for each $t\in (0,T)$. 
\end{prop}

\begin{proof}
By direct computation, we have
\begin{align*}
\frac{d}{dt}\int_U \phi F(D\bv)dx
&=\int_U \left(\phi_t F(D\bv) + \phi DF(D\bv)\cdot D\bv_t\right)dx\\
&=\int_U \left( \phi_t F(D\bv) - \Div(\phi DF(D\bv))\cdot \bv_t\right)dx\\
&=\int_U \left(\phi_t F(D\bv) - DF(D\bv) D\phi\cdot \bv_t - \phi D\psi(\bv_t)\cdot \bv_t\right)dx.
\end{align*} 
\end{proof}
Let us now assume
\begin{equation}\label{psiFzeroAtzero}
\psi(0)=|D\psi(0)|=F(O)=|DF(O)|=0.
\end{equation}
Here $O$ is the $m\times n$ matrix of zeros.  This assumption can be made without loss of generality. Note that we can choose $a\in \R^n$ such that $\inf_{\R^m}\psi=\psi(a)$, and set $\tilde{\psi}(w):=\psi(w+a)-\psi(a)$ and $\tilde{F}(M):=F(M)-(F(O)+DF(O)\cdot M)$. Then $\tilde{\bv}:=\bv -a t$ satisfies
\be\label{tildemainPDE}
D\tilde{\psi}(\tilde{\bv}_t)=\Div D\tilde{F}(D\tilde{\bv})
\ee
and $\tilde{\psi}$ and $\tilde{F}$ satisfy \eqref{UnifConv}, \eqref{UnifConv2} and  \eqref{psiFzeroAtzero}.  Moreover, if $\tilde{\bv}$ is a classical solution of \eqref{tildemainPDE}, then $\bv$ is a classical solution of \eqref{mainPDE}.

\par It now follows from  \eqref{UnifConv} that 
\begin{equation}\label{MoreEstPsi}
\begin{cases}
 \frac{1}{2}\theta|w|^2\le \psi(w)\le\frac{1}{2}\Theta |w|^2\\
 \frac{1}{2}\theta|w|^2\le D\psi(w)\cdot w-\psi(w)\le\frac{1}{2}\Theta |w|^2\\
\theta |w|^2\le D\psi(w)\cdot w\le \Theta|w|^2\\
\theta |w|\le |D\psi(w)|\le\Theta |w|,
 \end{cases}
\end{equation}
and from \eqref{UnifConv2} that
\begin{equation}\label{MoreEstF}
\begin{cases}
 \frac{1}{2}\lambda|M|^2\le F(M)\le\frac{1}{2}\Lambda |M|^2\\
 \frac{1}{2}\lambda|w|^2\le DF(M)\cdot M-F(M)\le\frac{1}{2}\Lambda |M|^2\\
\lambda |M|^2\le DF(M)\cdot M\le \Lambda|M|^2\\
\lambda |M|\le |DF(M)|\le\Lambda |M|.
 \end{cases}
\end{equation}
We will make use of these simplifications to derive a useful estimate on smooth solutions of \eqref{mainPDE}.

% First energy estimate  
\begin{cor}\label{CorBound1}
Assume $\eta\in C^\infty_c(U\times(0,T))$ with $\eta\ge 0$ and $\bv\in C^\infty(U\times(0,T);\R^m)$ is a solution of \eqref{mainPDE}. Then 
there is a constant $C$ depending only on $\theta, \lambda, \Theta,$ and $\Lambda$ such that 
\begin{equation}\label{EnergyBound1}
\max_{0\le t\le T}\int_U\eta^2|D\bv|^2 dx + \int^T_0\int_U \eta^2|\bv_t|^2dxdt 
\le C\int^T_0\int_U\left(\eta|\eta_t|+|D\eta|^2\right) |D\bv|^2dxdt.
\end{equation}
\end{cor}
\begin{proof}  Choose $\phi=\eta^2$ in \eqref{MainIdentity}. Employing \eqref{UnifConv}, \eqref{MoreEstPsi} and \eqref{MoreEstF} and integrating \eqref{MainIdentity} over the interval $[0,t]$ gives  
\begin{align*}
&\frac{\lambda}{2}\int_U\eta(x,t)^2|D\bv(x,t)|^2 dx+\theta\int^t_0\int_U\eta^2|\bv_t|^2 dxds\\
&\hspace{1in}\le \int_U\eta(x,t)^2F(D\bv(x,t)) dx+\int^t_0\int_U\eta^2 D\psi(\bv_t)\cdot \bv_t dxds\\
&\hspace{1in}= \int^t_0\int_U\left(F(D\bv)2\eta\eta_t -\bv_t\cdot DF(D\bv)2\eta D\eta \right)dxds\\
&\hspace{1in}\le \int^t_0\int_U\left(\Lambda|D\bv|^2\eta|\eta_t| -\bv_t\cdot DF(D\bv)2\eta D\eta \right)dxds\\
&\hspace{1in}\le \int^t_0\int_U\left(\Lambda|D\bv|^2\eta|\eta_t| +\eta|\bv_t|\cdot 2\Lambda|D\bv||D\eta| \right)dxds\\
&\hspace{1in}= \int^t_0\int_U\left(\Lambda|D\bv|^2\eta|\eta_t| +\sqrt{\theta}\eta|\bv_t|\cdot \frac{2\Lambda}{\sqrt{\theta}}|D\bv||D\eta| \right)dxds\\
&\hspace{1in}\le \Lambda\left(1+\frac{2\Lambda}{\theta}\right)\int^t_0\int_U\left(\eta|\eta_t|+|D\eta|^2\right)|D\bv|^2 dxds+\frac{\theta}{2}\int^t_0\int_U\eta^2|\bv_t|^2 dxds.
\end{align*}
Consequently, the assertion holds with
$$
C=\frac{\displaystyle\Lambda\left(1+\frac{2\Lambda}{\theta}\right) }{\frac{1}{2}\min\{\theta,\lambda\}}.
$$
\end{proof}
We have one more identity to derive, which involves the second derivatives of $F$. Let ${\cal S}^2(\Mmn)$ denote the space of 
symmetric bilinear forms on $\Mmn\times\Mmn$. We define $D^2F(M)\in {\cal S}^2(\Mmn)$, the Hessian of $F$ at $M$, as
$$
D^2F(M): \Mmn\times \Mmn\rightarrow\R; 
(\xi,\zeta)\mapsto\sum^n_{j,\ell=1}\sum^m_{i,k=1}F_{M^i_jM^k_\ell}(M)\xi^i_j\zeta^k_\ell.
$$
For $\xi\in \Mmn$, we also define $D^2F(M)\xi:\Mmn\rightarrow \R$ as the linear form $\zeta\mapsto D^2F(M)(\xi,\zeta)$. 
Also note that we can re-express \eqref{UnifConv2} as 
\be\label{UnifConv3}
\lambda |\xi|^2\le D^2F(M)(\xi,\xi)\le\Lambda|\xi|^2\quad (M,\xi\in \Mmn).
\ee

\par Using this definition, we have the following identity for smooth solutions of \eqref{mainPDE}.

% Second identity 
\begin{prop}
Let $\psi^*$ denote the Legendre transform of $\psi$. Assume $\phi\in C^\infty_c(U\times(0,T))$ and $\bv\in C^\infty(U\times(0,T);\R^m)$ is a solution of \eqref{mainPDE}. Then 
\begin{align}\label{SecondIdentity}
&\frac{d}{dt}\int_U \phi\psi^*(D\psi(\bv_t))dx+\int_U\phi D^2F(D\bv)(D\bv_t, D\bv_t) dx=\\
&\hspace{1in}\int_U\left(\psi^*(D\psi(\bv_t))\phi_t -\bv_t\cdot D^2F(D\bv)D\bv_tD\phi \right)dx
\end{align}
for each $t\in (0,T)$. 
\end{prop}
\begin{proof}
Recall that $D\psi^*$ and $D\psi$ are inverse mappings of $\R^n$. We can use this fact and differentiate equation \eqref{mainPDE} in time to get 
\begin{align*}
\frac{d}{dt}\int_U \phi\psi^*(D\psi(\bv_t))dx&=\int_U\left(\phi_t\psi^*(D\psi(\bv_t)) +\phi \bv_t\cdot \partial_t(D\psi(\bv_t))\right)dx\\
&=\int_U\left(\phi_t\psi^*(D\psi(\bv_t)) + \phi\bv_t\cdot \partial_t(\Div(DF(D\bv))\right)dx\\
&=\int_U\left(\phi_t\psi^*(D\psi(\bv_t)) + \phi\bv_t\cdot \Div(D^2F(D\bv)D\bv_t)\right)dx\\
&=\int_U\left(\phi_t\psi^*(D\psi(\bv_t)) -\bv_t\cdot D^2F(D\bv)D\bv_tD\phi -\phi D^2F(D\bv)(D\bv_t, D\bv_t)\right)dx.
\end{align*} 
\end{proof}

% second energy estimate  
\begin{cor}
Assume $\eta\in C^\infty_c(U\times(0,T))$ with $\eta\ge 0$ and $\bv\in C^\infty(U\times(0,T);\R^m)$ is a solution of \eqref{mainPDE}. Then 
there is a constant $C$ depending only on $\theta, \lambda, \Theta,$ and $\Lambda$ such that 
\begin{equation}\label{EnergyBound2}
\max_{0\le t\le T}\int_U\eta^2|\bv_t|^2 dx + \int^T_0\int_U \eta^2|D\bv_t|^2dxdt 
\le C\int^T_0\int_U\left(\eta|\eta_t|+|D\eta|^2\right) |\bv_t|^2dxdt.
\end{equation}
\end{cor}

\begin{proof}
We will choose $\phi=\eta^2$ in \eqref{SecondIdentity}. Observe, 
\begin{align*}
&\frac{\theta}{2}\int_U\eta(x,t)^2|\bv_t(x,t)|^2 dx+\lambda\int^t_0\int_U\eta^2|D\bv_t|^2 dxds\\
&\hspace{.7in}\le \int_U \eta(x,t)^2\psi^*(D\psi(\bv_t(x,t)))dx+\int^t_0\int_U\eta^2 D^2F(D\bv)(D\bv_t, D\bv_t) dxds\\
&\hspace{.7in}=\int^t_0\int_U\left(\psi^*(D\psi(\bv_t))2\eta\eta_t -\bv_t\cdot D^2F(D\bv)D\bv_t2\eta D\eta \right)dxds\\
&\hspace{.7in}\le\int^t_0\int_U\left(\Theta|\bv_t|^2\eta|\eta_t| +|\bv_t|\Lambda|D\bv_t|2\eta |D\eta| \right)dxds\\
&\hspace{.7in}=\int^t_0\int_U\left(\Theta|\bv_t|^2\eta|\eta_t| +\frac{2\Lambda}{\sqrt{\lambda}}|\bv_t||D\eta|\cdot  \sqrt{\lambda}\eta|D\bv_t| \right)dxds\\
&\hspace{.7in}=\left(\Theta+\frac{2\Lambda^2}{\lambda}\right)\int^t_0\int_U\left(\eta|\eta_t|+|D\eta|^2\right)|\bv_t|^2dxds +\frac{\lambda}{2}\int^t_0\int_U\eta^2|D\bv_t|^2 dxds.
\end{align*}
So we can take 
$$
C=\frac{\displaystyle\Theta+\frac{2\Lambda^2}{\lambda}}{\frac{1}{2}\min\{\theta,\lambda\}}.
$$
\end{proof}

%%%%%%%%%%%%%%%%%%%%%%%%%%%%%%%%%%%%%%%%%%%%%%%%%%%%%%%%%
\section{Weak Solutions}
The estimates \eqref{EnergyBound1} and \eqref{EnergyBound2} lead us to the following definition of a weak solution of \eqref{mainPDE}.  
Note carefully that we will make this definition for $\psi$ and $F$ that are only assumed to be continuously differentiable because this is the minimal regularity the definition requires. Otherwise (and aside from Proposition \ref{CompactnessProp} below) we will assume that $\psi$ and $F$ are twice continuously 
differentiable. 
 
% definition 
\begin{defn}\label{WeakSolnDefn}
Let $\psi\in C^1(\R^m)$ and $F\in C^1(\Mmn)$. A mapping $\bv\in L^2_{\text{loc}}(U\times (0,T);\R^m)$ is a {\it weak solution} of \eqref{mainPDE} on $U\times (0,T)$ if it satisfies 
\begin{equation}\label{NaturalSpace}
\begin{cases}
D\bv\in L^\infty_{\text{loc}}((0,T); L^{2}_{\text{loc}}(U;\Mmn))\\\\
\bv_t\in L^\infty_{\text{loc}}((0,T); L^2_{\text{loc}}(U;\R^m))\\\\
D\bv_t\in L^2_{\text{loc}}(U\times(0,T); \Mmn),
\end{cases}
\end{equation}
and
\begin{equation}\label{WeakSolnCond}
\int_U D\psi(\bv_t(x,t))\cdot  \bw(x)dx + \int_U DF(D\bv(x,t))\cdot D\bw(x)dx=0
\end{equation}
for all $\bw\in H^{1}_0(U;\R^m)$ and almost every $t\in (0,T)$.
\end{defn}

% Notation 
\par Recall that for an open subset $V\subset\R^n$, the Sobolev space $H^1_0(V; \R^m)$ is defined as the closure of $C^\infty_c(V;\R^m)$ in the norm 
$$
\|\bu\|_{H^1_0(V; \R^m)}:=\left(\int_V|D\bu|^2dx\right)^{1/2}.
$$
Moreover, its continuous dual space is $H^{-1}(V; \R^m):=\left(H^1_0(V; \R^m)\right)^*$.  Also recall that for a given Hilbert space $H$,
$$
AC^2([t_0,t_1]; H)
$$
is the space of paths $\gamma:[t_0,t_1]\rightarrow H$ that are differentiable almost everywhere 
with $\dot\gamma\in L^2([t_0,t_1]; H)$.  In particular, these paths are absolutely continuous and the fundamental theorem of calculus holds for such paths (Remark 1.1.3 \cite{AGS}).  

% weak continuity 

\par Our first of several results involving the integrability and continuity properties of weak solutions is as follows.  In this and subsequent assertions, we will identify equivalence classes of integrable mappings with their continuous representatives whenever it is possible to do so. 

\begin{prop}\label{WeakContProp}
Assume $\bv$ is a weak solution of \eqref{mainPDE}. Suppose $[t_0,t_1]\subset (0,T)$ and $V\subset\subset U$ is open. Then 
$\bv\in AC^2([t_0,t_1]; H^1(V;\R^m))$ and $D\psi(\bv_t)\in AC^2([t_0,t_1]; H^{-1}(V;\R^m))$.  
\end{prop}

\begin{proof}
As $\bv_t\in L^2([t_0,t_1];L^2(V;\R^m))$, $\bv\in AC^2([t_0,t_1]; L^2(V;\R^m))$; and since  \\ $D\bv_t\in L^2([t_0,t_1];L^2(V;\Mmn))$, $D\bv\in AC^2([t_0,t_1]; L^2(V;\Mmn))$. It follows that $\bv\in AC^2([t_0,t_1]; H^1(V;\R^m))$.  Moreover, in view of \eqref{WeakSolnCond}, 
\begin{align}\label{DpsiDvtSlope}
\|D\psi(\bv_t(\cdot,t))-D\psi(\bv_t(\cdot,s))\|_{H^{-1}(V;\R^m)}&\le \|DF(D\bv(\cdot,t))-DF(D\bv(\cdot,s))\|_{L^{2}(V;\Mmn)}\nonumber \\
&\le \Lambda \|D\bv(\cdot,t)-D\bv(\cdot,s)\|_{L^{2}(V;\Mmn)}. 
\end{align}
Thus, $D\psi(\bv_t): [t_0,t_1]\rightarrow H^{-1}(V;\R^m)$ is absolutely continuous as asserted. 
\end{proof}

\par Employing regularity results for elliptic PDE, we can deduce that the third spatial derivatives of weak solutions are locally square integrable in space and time.  To this end, we will denote ${\cal S}^2(\R^n;\R^m)$ for symmetric bilinear mappings from 
$(\R^n)^2\rightarrow \R^m$. For $V\subset U$ open and $\bu\in H^2(V;\R^m)$, $D^2\bu \in L^2(V;{\cal S}^2(\R^n;\R^m))$ is
defined as
\be\label{HessianOfbu}
D^2\bu(x)(z^1,z^2):=\sum^n_{i,j=1}\bu_{x_ix_j}(x)z^1_iz^2_j\quad (x\in V, z^1,z^2\in \R^n).
\ee
Here, of course, $\bu_{x_ix_j}\in L^2(V;\R^m)$ is a weak second partial derivative of $\bu$ for each $i,j=1,\dots,n$. We will also write 
$$
|D^2\bu|^2:=\sum^n_{i,j=1}|\bu_{x_ix_j}|^2.
$$
We can just as easily define $D^3\bu\in L^2(V;{\cal S}^3(\R^n;\R^m))$ in terms of the weak partial derivatives $\bu_{x_ix_jx_k}$ for any mapping 
$\bu \in H^3(V;\R^m)$. Here ${\cal S}^3(\R^n;\R^m)$ is the space of trilinear mappings on $(\R^n)^3$ with values in $\R^m$.

% H^2 regularity 
\begin{prop}\label{H2prop}
Assume $\bv$ is a weak solution of \eqref{mainPDE} on $U\times(0,T)$. \\
(i) Then 
\begin{equation}\label{H2Boundv}
D^2\bv\in L^\infty_{\text{loc}}((0,T); L^2_\text{loc}(U; {\cal S}^2(\R^n;\R^m))).
\end{equation}
In particular, equation \eqref{mainPDE} holds almost everywhere in $U\times(0,T)$. \\
(ii) Moreover,
\be\label{H3Boundv}
D^3\bv\in L^2_{\text{loc}}(U\times(0,T); {\cal S}^3(\R^n;\R^m)).
\ee
\end{prop}
\begin{proof}
$(i)$ By \eqref{WeakSolnCond}, 
\begin{equation}\label{MainEqnTimeSlice}
D\psi(\bv_t(\cdot,t))=\Div DF(D\bv(\cdot,t))
\end{equation}
weakly in $U$ for almost every $t\in (0,T)$. Now let $V,W\subset U$ be open with $V\subset\subset W\subset\subset U$. 
As $\bv(\cdot,t)$ satisfies \eqref{MainEqnTimeSlice}, the associated $W^{2,2}_{\text{loc}}(U)$ estimates (Proposition 8.6 in \cite{GiaMar} or Theorem 1, Section 8.3 of \cite{Evans}) for uniformly elliptic Euler-Lagrange equations imply $D^2\bv(\cdot, t)\in L^2(V; {\cal S}^2(\R^n;\R^m))$ and
\begin{align}\label{H2EstimateSpec}
\int_{V} |D^2\bv(x,t)|^2dx&\le C\int_W\left(|D\psi(\bv_t(\cdot,t))|^2+|D\bv(x,t)|^2\right)dx\nonumber \\
&\le C\int_W\left(\Lambda^2|\bv_t(x,t)|^2+|D\bv(x,t)|^2\right)dx
\end{align}
for almost every $t\in (0,T)$. Here $C$ is a constant that is independent of $\bv$. The assertion \eqref{H2Boundv} now follows from recalling \eqref{NaturalSpace} and taking the essential supremum in the above inequality locally in time.

\par Now that we have also established \eqref{H2Boundv}, we can integrate by parts in \eqref{WeakSolnCond} to get 
$$
\int_U \left[D\psi(\bv_t(x,t))-\Div(DF(D\bv(x,t)))\right]\cdot \bw(x) dx=0,
$$
for all $\bw\in H^1_0(U; \R^m)$ and almost every time $t\in (0,T)$. Thus $D\psi(\bv_t) =\Div DF(D\bv)$ almost everywhere in $U\times(0,T)$. 

\par $(ii)$ For almost every time $t\in (0,T)$, we again recall that \eqref{MainEqnTimeSlice} holds. We also have that $D\bv_t\in L^2_{\text{loc}}(U\times (0,T);\Mmn)$. 
Using difference quotients (as defined in Chapter 4 of \cite{GiaMar} or Chapter 5 of \cite{Evans}), we can differentiate \eqref{MainEqnTimeSlice} with respect to $x_i$ and show that 
\begin{equation}\label{MainEqnTimeSliceXi}
D^2\psi(\bv_t(\cdot,t))\bv_{tx_i}(\cdot,t)=\Div\left[D^2F(D\bv(\cdot,t))D\bv_{x_i}(\cdot,t)\right]
\end{equation}
holds weakly in $U$ for almost every $t\in (0,T)$ and each $i=1,\dots, n$. It also follows that $\bv_{x_i}(\cdot,t)\in H^2_{\text{loc}}(U;\R^m)$ for almost every $t\in (0,T)$ and 
\begin{align}\label{H3EstimateYeah}
\int_{V} |D^2\bv_{x_i}(x,t)|^2dx&\le C\int_W\left(|D^2\psi(\bv_t(x,t))\bv_{tx_i}(x,t)|^2+|D\bv_{x_i}(x,t)|^2\right)dx\\
&\le C\int_W\left(\Lambda^2|D\bv_t(x,t)|^2+|D^2\bv(x,t)|^2\right)dx.
\end{align}
The assertion follows by integrating this inequality locally in time. 
\end{proof}

We can now establish an improved higher space-time integrability of $D^2\bv$ and $\bv_t$.

\begin{cor}\label{PeeCorollary}
Assume $\bv$ is a weak solution of \eqref{mainPDE} on $U\times(0,T)$. There exists an exponent $p>2$ such that
\begin{equation}\label{extraIntegrablePee}
D^2\bv\in L^{p}_{\text{loc}}(U\times(0,T); {\cal S}^2(\R^n;\R^m) )
\end{equation}
and
\begin{equation}\label{vtIntegrablePee}
\bv_t\in L^{p}_{\text{loc}}(U\times(0,T);\R^m).
\end{equation}
\end{cor}

\begin{proof}
By the interpolation of the Lebesgue spaces and the Gagilardo-Nirenberg-Sobolev inequality, we have the inclusion
$$
L^\infty_{\text{loc}}((0,T); L^2_\text{loc}(U))\cap L^2_{\text{loc}}((0,T); H^1_\text{loc}(U))\subset L^{2+\frac{4}{n}}_{\text{loc}}(U\times (0,T))
$$
for $n\ge 3$.  For $n=2$, we also have 
$$
L^\infty_{\text{loc}}((0,T); L^2_\text{loc}(U))\cap L^2_{\text{loc}}((0,T); H^1_\text{loc}(U))\subset L^{q}_{\text{loc}}(U\times (0,T))
$$
for each $1\le q<2+\frac{4}{n}=4$.  See Corollary 3.4 of \cite{Hynd} and Lemma 5.3 of \cite{DuzMinSte}.

\par For $n=1$, let us suppose 
$$
w\in L^\infty_{\text{loc}}((0,T); L^2_\text{loc}(U))\cap L^2_{\text{loc}}((0,T); H^1_\text{loc}(U)),
$$
$V\subset\subset U$ is an interval, and $[t_0,t_1]\in (0,T)$. We have
$$
w(x,t)=w(y,t)+\int^x_yw_{x}(z,t)dz
$$
for $x,y\in V$ and almost every $t\in [t_0,t_1]$. It then follows that
$$
w(x,t)^2\le 2\left(w(y,t)^2+|V|\int_Vw_x(z,t)^2dz\right).
$$
Integrating over $y\in V$ gives 
$$
|V|w(x,t)^2\le 2\left(\int_Vw(y,t)^2dy+|V|^2\int_Vw_{x}(z,t)^2dz\right).
$$

\par Therefore,  $w\in L^2([t_0,t_1]; L^\infty(V))$. Combining with the fact that $w\in L^{\infty}([t_0,t_1];L^2(V))$ gives
\begin{align*}
\int^{t_1}_{t_0}\int_Vw(x,t)^4dxdt&=\int^{t_1}_{t_0}\int_Vw(x,t)^2w(x,t)^2dxdt\\
&\le \int^{t_1}_{t_0}|w(\cdot,t)|_{L^\infty(V;\R^m)}^2 \left(\int_Vw(x,t)^2dx\right)dt\\
&\le \left(\int^{t_1}_{t_0}|w(\cdot,t)|_{L^\infty(V;\R^m)}^2dt\right) \left(\displaystyle\underset{t\in [t_0,t_1]}{\text{ess sup}}\int_Vw(x,t)^2dx\right)\\
&<\infty.
\end{align*}
Thus, $w\in L^4_{\text{loc}}(U\times(0,T))$.

\par So for every $n\ge 1$
$$
L^\infty_{\text{loc}}((0,T); L^2_\text{loc}(U))\cap L^2_{\text{loc}}((0,T); H^1_\text{loc}(U))\subset L^{p}_{\text{loc}}(U\times (0,T))
$$
for some $p>2$. The claim follows as our arguments apply to $w= v^i_t $ by \eqref{NaturalSpace} and to $w=v^i_{x_jx_k}$ by Proposition \ref{H2prop} for each $i=1,\dots, m$ and $j,k=1,\dots, n$. 
\end{proof}

% Verify integral identities 
Now we will show that the various identities and estimates we derived for smooth solutions actually hold for weak solutions. 
\begin{prop}
Assume $\bv$ is a weak solution of \eqref{mainPDE} in $U\times(0,T)$ and $\phi\in C^\infty_c(U\times (0,T))$. Then $(0,T)\mapsto \int_U F(D\bv(x,t))\phi(x,t) dx$ is  
absolutely continuous and \eqref{MainIdentity} holds for almost every $t\in (0,T)$. 
\end{prop}
\begin{proof}
1. Let $u\in C^\infty_c(U)$ be nonnegative and suppose $u$ is 
supported in an open set $V\subset\subset U$ with smooth boundary.  For $\bw\in L^2(V;\R^m)$, we define 
$$
\Phi(\bw):=
\begin{cases}
\int_V u F(D\bw)dx, \quad &\bw\in H^1(V;\R^m)\\
+\infty, \quad & \text{otherwise}.
\end{cases}
$$
Note that $\Phi$ is convex, lower-semicontinuous and proper. If $\bw\in H^2(V;\R^m)$, then it is routine to compute 
\begin{align*}
\partial\Phi(\bw)&:=\left\{ \xi\in L^2(V;\R^m): \Phi(\bu)\ge \Phi(\bw)+\int_V\xi\cdot (\bu-\bw)dx\;\; \text{all}\;\; \bu\in L^2(V;\R^m)\right\}\\
&=\{-\Div(uDF(D\bw))\}.
\end{align*}
In this case, we write 
$$
|\partial\Phi|(\bw):=\|-\Div(uDF(D\bw))\|_{L^2(V;\R^m)}.
$$

\par 2. Recall that a weak solution $\bv$ is absolutely continuous with values in $L^2(V;\R^m)$. By the $H^2_\text{loc}$ estimate \eqref{H2Boundv}, we have 
$$
|\partial\Phi|(\bv)\cdot \|\bv_t\|_{L^2(V;\R^m)}=\|-\Div(uDF(D\bv))\|_{L^2(V;\R^m)}\cdot\|\bv_t\|_{L^2(V;\R^m)}\in L^1_{\text{loc}}(0,T).
$$
In view of Proposition 1.4.4 and Remark 1.4.6 \cite{AGS}, $\Phi\circ\bv$ is locally absolutely continuous on $(0,T)$. By the chain rule (Remark 1.4.6 \cite{AGS}) and the weak solution condition \eqref{WeakSolnCond},
\begin{align*}
\frac{d}{dt}\Phi\circ\bv(\cdot,t)&=\frac{d}{dt}\int_V u(x) F(D\bv(x,t))dx\\
&=\frac{d}{dt}\int_U u(x) F(D\bv(x,t))dx\\
&=\int_U u(x) DF(D\bv(x,t))\cdot D\bv_t(x,t)dx\\
%&= -\int_V\Div(u(x)DF(D\bv(x,t)))\cdot \bv_t(x,t)dx \\
&= -\int_U\Div(u(x)DF(D\bv(x,t)))\cdot \bv_t(x,t)dx \\
&= -\int_U\left(DF(D\bv(x,t))Du(x)\cdot \bv_t(x,t)+u(x)D\psi(\bv_t(x,t))\cdot \bv_t(x,t)\right)dx
\end{align*}
for almost every $t\in (0,T)$.  As a result, 
\begin{align}\label{Identity1Unonneg}
&\int_U u(x) F(D\bv(x,t))dx=\int_U u(x) F(D\bv(x,s))dx\nonumber \\
&\hspace{1.5in} -\int^t_s\int_U(DF(D\bv)Du\cdot \bv_t+uD\psi(\bv_t)\cdot \bv_t)dxd\tau
\end{align}
for $0< s\le t<T$.

\par 3. Now suppose $u\in C^\infty_c(U)$ is not necessarily nonnegative.  Let $u^\epsilon=\eta^\epsilon*u$ denote the standard mollification of $u$ ($\epsilon>0)$. 
Recall that $\eta\in C^\infty_c(B_1(0))$ is a nonnegative, radial function that satisfies $\int_{B_1(0)}\eta(z)dz=1$ and $\eta^\epsilon:=\epsilon^{-n}\eta(\cdot/\epsilon)$. It is routine to check that for all $\epsilon>0$ sufficiently small, $u^\epsilon\in C^\infty_c(U)$ and that $u^\epsilon\rightarrow u$ in $C^1_\text{loc}(U)$ as $\epsilon\rightarrow 0^+$.  Decomposing $u$ into its positive and negative parts $u=u^+-u^-$, we have $u^\epsilon=(u^+)^\epsilon-(u^-)^\epsilon$. In particular, $(u^\pm)^\epsilon\in C^\infty_c(U)$ are both nonnegative. Therefore, \eqref{Identity1Unonneg} holds for 
$u=(u^\pm)^\epsilon$. Subtracting identity \eqref{Identity1Unonneg} with $u=(u^-)^\epsilon$ from the same identity \eqref{Identity1Unonneg} with $u=(u^+)^\epsilon$ gives 
\begin{align}\label{Identity1Unonneg}
&\int_U u^\epsilon(x) F(D\bv(x,t))dx=\int_U u^\epsilon(x) F(D\bv(x,s))dx\nonumber \\
&\hspace{1in} -\int^t_s\int_U(DF(D\bv)Du^\epsilon\cdot \bv_t+u^\epsilon D\psi(\bv_t)\cdot \bv_t)dxd\tau.\quad
\end{align}
Sending $\epsilon\rightarrow 0^+$ allows us to conclude \eqref{Identity1Unonneg} without any sign restriction on $u$. 

\par 4. Let us define $f(t):=  \int_U F(D\bv(x,t))\phi(x,t) dx$ and suppose $h\neq 0$. Note 
\begin{align*}
\frac{f(t+h)-f(t)}{h}&=
\int_U\frac{F(D\bv(x,t+h))\phi(x,t+h)-F(D\bv(x,t))\phi(x,t)}{h}dx\\
&=\int_U\phi(x,t)\left[\frac{F(D\bv(x,t+h))-F(D\bv(x,t))}{h}\right]dx \\
&\hspace{.5in}+ \int_U F(D\bv(x,t+h))\left[\frac{\phi(x,t+h)-\phi(x,t)}{h} \right]dx.
\end{align*}
By parts 2 and 3 above,
\begin{align*}
\lim_{h\rightarrow 0}\int_U\phi(x,t)\left[\frac{F(D\bv(x,t+h))-F(D\bv(x,t))}{h}\right]dx  = 
\hspace{1in} \\
-\int_U\left(DF(D\bv(x,t))D\phi(x,t)\cdot \bv_t(x,t)+\phi(x,t) D\psi(\bv_t(x,t))\cdot \bv_t(x,t)\right)dx
\end{align*}
for almost every $t\in (0,T)$.  In view of the continuity of $D\bv$ (as detailed in Proposition \ref{WeakContProp}), we also have
\begin{align*}
\lim_{h\rightarrow 0}\int_U F(D\bv(x,t+h))\left[\frac{\phi(x,t+h)-\phi(x,t)}{h} \right]dx=\int_U F(D\bv(x,t))\phi_t(x,t)dx
\end{align*}
for every $t\in (0,T)$.  Combining these limits completes a proof that \eqref{MainIdentity} holds for almost every $t\in (0,T)$. Finally, we note that if \eqref{MainIdentity} holds then $f$ is absolutely continuous as each term in \eqref{MainIdentity} aside from the time derivative belongs to $L^1[0,T]$.
\end{proof}
% Local energy inequality holds for weak solutions 
\begin{cor}
Every weak solution of \eqref{mainPDE} on $U\times (0,T)$ satisfies the local energy estimate \eqref{EnergyBound1}.
\end{cor}
Let us now proceed to establishing the identity \eqref{SecondIdentity} for weak solutions. This identity combined with the local boundedness of $\bv_t: (0,T)\rightarrow L^2(V;\R^m)$ will actually allow us to verify that $D\psi(\bv_t)$ is strongly continuous with values in $L^2(V;\R^m)$. We also remind the reader that $\psi^*$ is the Legendre transform of $\psi$. 

\begin{prop}
Assume $\bv$ is a weak solution of \eqref{mainPDE} in $U\times(0,T)$ and $\phi\in C^\infty_c(U\times (0,T))$. Then $(0,T)\mapsto \int_U \psi^*(D\psi(\bv_t(x,t)))\phi(x,t) dx$ is locally absolutely continuous and \eqref{SecondIdentity} holds for almost every $t\in (0,T)$. 
\end{prop}

\begin{proof}
1. Suppose $u\in C^\infty_c(U)$ is nonnegative and choose an open $V\subset\subset U$ such that $u$ is supported in $V$. Let us also define 
$$
\Psi(\bw):=
\begin{cases}
\int_Vu\psi^*(\bw)dx,\quad &\bw\in L^2(V;\R^m)\\
+\infty,\quad & \text{otherwise}
\end{cases}
$$
for each $\bw\in H^{-1}(V;\R^m)$.  Observe that $\Psi$ is convex, lower-semicontinuous and proper. A routine computation shows that if $D\psi^*(\bw)\in H^1(V;\R^m)$ then 
\begin{align*}
\partial\Psi(\bw):&=\left\{ \xi\in H^1_0(V;\R^m): \Psi(\bu)\ge \Psi(\bw)+\langle\xi,\bu-\bw\rangle\;\; \text{all}\;\; \bu\in H^{-1}(V;\R^m)\right\}\\
&=\{uD\psi^*(\bw)\},
\end{align*}
where $\langle,\rangle$ is the pairing between $H^1_0(V;\R^m)$ and $H^{-1}(V;\R^m)$.  In this case, we will write 
\be\label{PsiSlopeForm}
|\partial\Psi|(\bw):=\|uD\psi^*(\bw)\|_{H^1_0(V;\R^m)}.
\ee

\par 2. By Proposition \ref{WeakContProp}, $D\psi(\bv_t)\in AC^2([t_0,t_1]; H^{-1}(V;\R^m))$.  We also have by the weak solution condition \eqref{WeakSolnCond}
that for every $\bu\in H^1_0(V;\R^m)$ and almost every $t\in (0,T)$
\begin{align*}
\left\langle \bu, \partial_tD\psi(\bv_t(\cdot, t))\right\rangle &= \frac{d}{dt}\langle \bu, D\psi(\bv_t(\cdot, t))\rangle\\
&= \frac{d}{dt}\int_V D\psi(\bv_t(x, t))\cdot \bu(x) dx\\
&= \frac{d}{dt}\int_U D\psi(\bv_t(x, t))\cdot \bu(x) dx\\
&= -\frac{d}{dt}\int_U DF(D\bv(x, t))\cdot D\bu(x) dx\\
&= -\int_U D^2F(D\bv(x, t))(D\bv_t(x, t), D\bu(x)) dx.
\end{align*}
The last equality above can be justified by employing the Lipschitz continuity of $DF$ and using that $D\bv: (0,T)\rightarrow L^2(V;\Mmn)$ is locally absolutely continuous; we leave the details to the reader.   

\par It follows from this computation (and also from inequality \eqref{DpsiDvtSlope}) that  $\|\partial_t D\psi(\bv_t)\|_{H^{-1}(V;\R^m)}\le \Lambda \|D\bv\|_{L^2(V;\Mmn)}$. Combining with \eqref{PsiSlopeForm} gives
$$
|\partial\Psi|(D\psi(\bv_t))\cdot\|\partial_t D\psi(\bv_t)\|_{H^{-1}(V;\R^m)}\le \Lambda\|u\bv_t\|_{H^1_0(V;\R^m)} \cdot \|D\bv_t\|_{L^2(V;\Mmn)}\in L^1_{\text{loc}}(0,T).
$$
Consequently, $\Psi\circ (D\psi(\bv_t))$ is absolutely continuous (Remark 1.4.6 of \cite{AGS}); and by the chain rule,
\begin{align*}
\frac{d}{dt}\Psi(D\psi(\bv_t(\cdot,t)))&=\frac{d}{dt}\int_Vu(x)\psi^*\left(D\psi(\bv_t(x,t))\right)dx\\
&=\frac{d}{dt}\int_Uu(x)\psi^*\left(D\psi(\bv_t(x,t))\right)dx\\
&=\left\langle u\bv_t (\cdot,t), \partial_tD\psi(\bv_t(\cdot, t))\right\rangle\\
&= -\int_U D^2F(D\bv(x, t))(D\bv_t(x, t), D\left(u(x)\bv_t (x,t)\right)) dx\\
&=-\int_U u(x)D^2F(D\bv(x, t))(D\bv_t(x, t), D\bv_t (x,t))dx \\
&\quad - \int_U\bv_t(x,t)\cdot D^2F(D\bv(x, t))D\bv_t(x, t)Du(x)dx.
\end{align*}
In summary, for nonnegative $u\in C^\infty_c(U)$, we have 
\begin{align}\label{SecondIdentitywithU}
\frac{d}{dt}\int_Uu(x)\psi^*\left(D\psi(\bv_t(x,t))\right)dx&=-\int_U u(x)D^2F(D\bv(x, t))(D\bv_t(x, t), D\bv_t (x,t))dx \nonumber \\
&\quad - \int_U\bv_t(x,t)\cdot D^2F(D\bv(x, t))D\bv_t(x, t)Du(x)dx.\quad 
\end{align}

\par 3. We can establish formula \eqref{SecondIdentitywithU} for any $u\in C^\infty_c(U)$ without sign restriction by arguing similar to how 
we did in part 3 of the previous proposition.  We may also complete this proof as we did in part 4 of the previous proposition, provided we verify that $D\psi(\bv_t): [t_0,t_1]\rightarrow L^2(V;\R^m)$ is continuous for any open $V\subset\subset U$ and $[t_0,t_1]\subset (0,T)$. We will first show that  $D\psi(\bv_t): [t_0,t_1]\rightarrow L^2(V;\R^m)$ is weakly continuous. 

\par To this end, let us choose an open set $W$ such that $V\subset\subset W\subset\subset U$.  Now select $u\in C^\infty_c(U)$ with $0\le u\le 1$, $u\equiv 1$ on $V$ and $u$ supported in $W$.  Let us also select a time $\tau\in [t_0,t_1]$ for which $\bv_t(\cdot, \tau)\in L^2(W;\R^m)$; such a time $\tau$ exists as $\bv_t\in L^2([t_0,t_1];L^2(W,\R^m)).$ By \eqref{MoreEstPsi} and \eqref{SecondIdentitywithU}, 
\begin{align*}
\frac{\theta}{2}\int_V|\bv_t(x,t)|^2dx&\le\int_V\psi^*\left(D\psi(\bv_t(x,t))\right)dx\\
&\le\int_Wu(x)\psi^*\left(D\psi(\bv_t(x,t))\right)dx\\
&=\int_Wu(x)\psi^*\left(D\psi(\bv_t(x,\tau))\right)dx +\\
&\quad -\int^t_\tau\int_W (uD^2F(D\bv)(D\bv_t, D\bv_t) +\bv_t\cdot D^2F(D\bv)D\bv_tDu)dxds\\
&\le \frac{\Theta}{2}\int_W|\bv_t(x,\tau)|^2dx+\Lambda\int^{t_1}_{t_0}\int_{W}\left( |D\bv_t|^2+|\bv_t||D\bv_t||Du|\right)dxds
\end{align*}
for each $t\in [\tau,t_1]$. We can derive a similar estimate for $t\in [t_0,\tau]$, and in view of \eqref{NaturalSpace}, we may conclude that $\int_V|\bv_t(x,t)|^2dx$ is uniformly bounded in $t\in [t_0,t_1]$.

\par  Recall  that $D\psi(\bv_t):[t_0,t_1]\rightarrow H^{-1}(V;\R^m)$ is continuous by Proposition \ref{WeakContProp}. Assume $(s_k)_{k\in \N}\subset [t_0,t_1]$ and $s_k\rightarrow t\in [t_0,t_1]$ as $k\rightarrow\infty$. As $|D\psi(\bv_t(\cdot,s_k))|\le \Theta |\bv_t(\cdot,s_k)|$, 
$\{D\psi(\bv_t(\cdot,s_k))\}_{k\in \N}\subset L^2(V;\R^m)$ is bounded. So there is a subsequence  $\{D\psi(\bv_t(\cdot,s_{k_j}))\}_{j\in \N}$ that converges weakly
to some $\bu$ in $L^2(V;\R^m)$. Since  $D\psi(\bv_t(\cdot,s_k))\rightarrow D\psi(\bv_t(\cdot,t))$ in $H^{-1}(V;\R^m)$, it must be that $\bu=D\psi(\bv_t(\cdot,t))$. And 
as this limit is independent of the subsequence, $D\psi(\bv_t(\cdot,s_k))\rightharpoonup D\psi(\bv_t(\cdot,t))$ in $L^{2}(V;\R^m)$. Clearly this argument extends to any bounded domain within $U$, 
so we actually have that $D\psi(\bv_t) : [t_0,t_1]\rightarrow L^2(W;\R^m)$ is weakly continuous. 

\par  By the uniform convexity of $\psi^*$, 
\begin{align*}
\int_W u(x)\psi^*(D\psi(\bv_t(x,s)))dx &\ge \int_W u(x)\psi^*(D\psi(\bv_t(x,t)))dx \\
&\quad+\int_W u(x)\bv_t(x,t)\cdot (D\psi(\bv_t(x,s))-D\psi(\bv_t(x,t)))dx\\
&\quad+\frac{1}{2\Theta}\int_V\left|D\psi(\bv_t(x,s))-D\psi(\bv_t(x,t))\right|^2dx
\end{align*}
for each $t,s\in [t_0,t_1]$. We can then use \eqref{SecondIdentitywithU} and the weak continuity of $D\psi(\bv_t)$ to get 
$$
\lim_{s\rightarrow t}\int_V\left|D\psi(\bv_t(x,s))-D\psi(\bv_t(x,t))\right|^2dx=0.
$$
As a result, we can now proceed as we did in part 4 of the previous proposition to verify $(0,T)\mapsto \int_U \psi^*(D\psi(\bv_t(x,t)))\phi(x,t) dx$ is locally absolutely continuous and \eqref{SecondIdentity} holds for almost every $t\in (0,T)$.  We leave the details to the reader. 
\end{proof}

\begin{cor}
Every weak solution of \eqref{mainPDE} on $U\times (0,T)$ satisfies the local energy estimate \eqref{EnergyBound2}.
\end{cor}
% Continuity of Dpsi(v_t) in L^2
\begin{cor}
Assume $\bv$ is a weak solution of \eqref{mainPDE} in $U\times(0,T)$. Suppose $[t_0,t_1]\subset (0,T)$ and $V\subset\subset U$ is open. Then 
 $\bv_t: [t_0,t_1]\rightarrow L^{2}(V;\R^m)$ is continuous. 
\end{cor}
\begin{proof}
In part 3 of the previous proposition, we established  that $D\psi(\bv_t): [t_0,t_1]\rightarrow L^{2}(V;\R^m)$ is continuous. In view of \eqref{UnifConv2}, 
$$
\int_V|\bv_t(x,t)-\bv_t(x,s)|^2dx\le \frac{1}{\theta^2}\int_V|D\psi(\bv_t(x,t))-D\psi(\bv_t(x,s))|^2dx.
$$
As a result, $\bv_t: [t_0,t_1]\rightarrow L^{2}(V;\R^m)$ is necessarily continuous. 
\end{proof}

%%%%%%%%%%%%%%%%%%%%%%%%%%%%%%%%%%%%%%%%%%%%%%%%%%%%%%%%%
\section{Fractional time differentiability}
We seek to strengthen our integrability and continuity assertions obtained in the previous section. In particular, 
we will derive some averaged continuity estimates for $\bv_t: [t_0,t_1]\rightarrow L^2(V;\R^m)$ and 
$D^2\bv: [t_0,t_1]\rightarrow L^2(V;{\cal S}^2(\R^n,\R^m))$, where $[t_0,t_1]\subset (0,T)$ and $V\subset\subset U$. As we shall see, these estimates imply a certain fractional time differentiability of these mappings.  As an application, we will use these estimates to 
derive compactness properties of solutions which play a crucial role in our proof of Theorem \ref{mainThm}. 

\begin{prop}
Assume $\bv$ is a weak solution of \eqref{mainPDE} in $U\times (0,T)$, and let $p>2$ be the exponent in Corollary \ref{PeeCorollary}.  For each open $V\subset\subset U$ and $[t_0,t_1]\in (0,T)$, there is a constant $C$ such that 
\begin{equation}\label{FractionVtDeriv}
\int^{t_1}_{t_0}\int_V|\bv_t(x,t+h)-\bv_t(x,t)|^2dxdt\le C |h|^{\frac{1}{2}-\frac{1}{p}}
\end{equation} 
for $0<|h|<\frac{1}{2}\min\{1,t_0,T-t_1\}$. 
\end{prop}

\begin{proof}
Assume $u\in C^\infty_c(U)$ is nonnegative and $u\equiv 1$ in $V$. Let us also initially suppose $0<h<\frac{1}{2}\min\{1,t_0,T-t_1\}$. By \eqref{SecondIdentitywithU}, we have 
\begin{align}
&\int_Uu(x)\psi^*\left(D\psi(\bv_t(x,t+h))\right)dx=\int_Uu(x)\psi^*\left(D\psi(\bv_t(x,t))\right)dx \nonumber \\
&-\int^{t+h}_t\int_U\left(uD^2F(D\bv)(D\bv_t, D\bv_t)+\bv_t\cdot D^2F(D\bv)D\bv_tDu\right)dxd\tau\\
&\le \int_Uu(x)\psi^*\left(D\psi(\bv_t(x,t))\right)dx-\int^{t+h}_t\int_U\bv_t\cdot D^2F(D\bv)D\bv_tDu dxd\tau.
\end{align}
Using the uniform convexity of $\psi^*$, we also have
\begin{align*}
-\int^{t+h}_t\int_U\bv_t\cdot D^2F(D\bv)D\bv_tDu dxd\tau&\ge \int_Uu(x)\left[\psi^*\left(D\psi(\bv_t(x,t+h))\right)-\psi^*\left(D\psi(\bv_t(x,t))\right)\right]dx\\
&\ge \int_Uu(x)\bv_t(x,t)\cdot\left(D\psi(\bv_t(x,t+h))-D\psi(\bv_t(x,t)))\right)dx\\
& +\frac{1}{2\Theta}\int_U u(x)\left|D\psi(\bv_t(x,t+h))-D\psi(\bv_t(x,t))\right|^2dx\\
&\ge \int_U\left(DF(D\bv(x,t+h))-DF(D\bv(x,t))\right)\cdot D(u(x)\bv_t(x,t)) dx\\
& +\frac{1}{2\Theta}\int_V\left|D\psi(\bv_t(x,t+h))-D\psi(\bv_t(x,t))\right|^2dx.
\end{align*}

\par Notice 
\begin{align*}
&\left|\int_U\left(DF(D\bv(x,t+h))-DF(D\bv(x,t))\right)\cdot D(u(x)\bv_t(x,t)) dx \right|\\
&\le \left(\int_U\left|DF(D\bv(x,t+h))-DF(D\bv(x,t))\right|^2dx \right)^{1/2}\left(\int_U|D(u\bv_t)|^2dx\right)^{1/2} \\
&\le\Lambda\left(\int_U\left|D\bv(x,t+h)-D\bv(x,t)\right|^2dx \right)^{1/2}\left(\int_U|D(u\bv_t)|^2dx\right)^{1/2} \\
&\le\Lambda h^{1/2}\left(\int^{t+h}_t\int_U\left|D\bv_t\right|^2dxds\right)^{1/2}\left(\int_U|\bv_t\otimes Du+uD\bv_t|^2dx\right)^{1/2}\\
&\le\Lambda h^{1/2}\left(\int^{(t_1+T)/2}_{t_0}\int_U\left|D\bv_t\right|^2dxds\right)^{1/2}\left(\int_U|\bv_t\otimes Du+uD\bv_t|^2dx\right)^{1/2}.
\end{align*}
Consequently, 
\begin{align*}
&\int^{t_1}_{t_0}\left|\int_U\left(DF(D\bv(x,t+h))-DF(D\bv(x,t))\right)\cdot D(u(x)\bv_t(x,t))dx\right|dt\\
&\le \Lambda h^{1/2}\left(\int^{(t_1+T)/2}_{t_0}\int_U\left|D\bv_t\right|^2dxds\right)^{1/2}\int^{t_1}_{t_0}\left(\int_U|\bv_t\otimes Du+uD\bv_t|^2dx\right)^{1/2}dt\\
&\le \Lambda (t_1-t_0)^{1/2} h^{1/2}\left(\int^{(t_1+T)/2}_{t_0}\int_U\left|D\bv_t\right|^2dxds\right)^{1/2}\left(\int^{t_1}_{t_0}\int_U|\bv_t\otimes Du+uD\bv_t|^2dxdt\right)^{1/2}.
\end{align*}
Here $\bv_t\otimes Du$ is the matrix value mapping with $i,j$th component function $v^i_tu_{x_j}$.

\par By Proposition \ref{H2prop} and Corollary \ref{PeeCorollary}, $|\bv_t||D\bv_t|\in L^{\frac{2p}{p+2}}_{\text{loc}}(U\times(0,T))$. 
As a result,
\begin{align*}
\int^{t+h}_t\int_U\bv_t\cdot D^2F(D\bv)D\bv_tDu dxd\tau&\le \Lambda \left(\int^{t+h}_t\int_V\left(|Du||\bv_t||D\bv_t|\right)^{\frac{2p}{p+2}}dxd\tau\right)^{\frac{1}{2}+\frac{1}{p}}(|V|h)^{\frac{1}{2}-\frac{1}{p}}\\
&\le\Lambda \left(\int^{(t_1+T)/2}_{t_0}\int_U\left(|Du||\bv_t||D\bv_t|\right)^{\frac{2p}{p+2}}dxd\tau\right)^{\frac{1}{2}+\frac{1}{p}}(|V|h)^{\frac{1}{2}-\frac{1}{p}}.
\end{align*}
Moreover, 
\begin{align*}
&\int^{t_1}_{t_0}\int^{t+h}_t\int_U\bv_t\cdot D^2F(D\bv)D\bv_tDu dxdt\\
&\hspace{1in}\le \Lambda(t_1-t_0)\left(\int^{(t_1+T)/2}_{t_0}\int_U\left(|Du||\bv_t||D\bv_t|\right)^{\frac{2p}{p+2}}dxd\tau\right)^{\frac{1}{2}+\frac{1}{p}}(|V|h)^{\frac{1}{2}-\frac{1}{p}}.
\end{align*}
Putting all of these inequalities together gives us 
\begin{align}\label{SpecificVtFractEstimate}
&\int^{t_1}_{t_0}\int_V\left|\bv_t(x,t+h)-\bv_t(x,t)\right|^2dxdt  \\
&\le\frac{1}{\theta^2}\int^{t_1}_{t_0}\int_V\left|D\psi(\bv_t(x,t+h))-D\psi(\bv_t(x,t))\right|^2dxdt \\
&\le\frac{2\Theta\Lambda}{\theta^2}  \left\{(t_1-t_0)^{1/2}\left(\int^{(t_1+T)/2}_{t_0}\int_U\left|D\bv_t\right|^2dxds\right)^{1/2}\left(\int^{t_1}_{t_0}\int_U|\bv_t\otimes Du+uD\bv_t|^2dxdt\right)^{1/2}
\right. \\
& \left.\quad+\; (t_1-t_0)\left(\int^{(t_1+T)/2}_{t_0}\int_U\left(|Du||\bv_t||D\bv_t|\right)^{\frac{2p}{p+2}}dxd\tau\right)^{\frac{1}{2}+\frac{1}{p}}|V|^{\frac{1}{2}-\frac{1}{p}}\right\}h^{\frac{1}{2}-\frac{1}{p}}.
\end{align}
A similar argument can be employed to establish \eqref{FractionVtDeriv} for $h<0$, as well. 
\end{proof}

We have the following consequence of the preceding proposition which asserts that $\bv_t$ is fractionally time differentiable as exhibited in \eqref{FractionVtDeriv2} below. We will omit a proof as this has been previously established (Proposition 3.4  \cite{DuzMin} or Proposition 2.19 of \cite{DuzMinSte}).
\begin{cor}\label{TimeDiffDVCor}
Assume $\bv$ is a weak solution of \eqref{mainPDE} in $U\times (0,T)$ and $p>2$ is the exponent in \eqref{extraIntegrablePee}.  For each open $V\subset\subset U$, $[t_0,t_1]\in (0,T)$, and
\begin{equation}\label{BetaInterval}
\beta\in \left(0,\frac{1}{2}-\frac{1}{p}\right),
\end{equation}
there is constant $A=A(p,\beta,t_0,t_1,V)>0$ such that 
\begin{equation}\label{FractionVtDeriv2}
\int^{t_1}_{t_0}\int^{t_1}_{t_0}\int_V\frac{|\bv_t(x,t)-\bv_t(x,s)|^2}{|t-s|^{1+\beta}}dxdtds\le A\left(C+\|\bv_t\|^2_{L^2(V\times[t_0,t_1])}\right).
\end{equation} 
Here $C$ is the constant in \eqref{FractionVtDeriv}.
\end{cor}
Now let us move on to establishing an analogous fractional time differentiability of $D^2\bv$.  One of the hypotheses of 
the following assertion (and of Theorem \ref{mainThm}) is that  
\be\label{DtwoFCalpha}
D^2F\in C^\alpha(\Mmn; {\cal S}^2(\Mmn))
\ee
for some $\alpha\in (0,1]$.   The reason we have decided to discuss this assumption prior to the statement is to emphasize that 
\eqref{DtwoFCalpha} also holds with any H\"older exponent less that or equal to $\alpha$, as well.  This claim follows as $D^2F$ is uniformly bounded (recall 
\eqref{UnifConv3}). Therefore, we can suppose without any loss of generality that \eqref{DtwoFCalpha} holds for an exponent 
$\alpha\in (0,1]$ that additionally satisfies
\be\label{alphaExtra}
\alpha\left(\frac{p}{p-2}\right)\le 2.
\ee
Here $p>2$ is the exponent in Corollary \ref{PeeCorollary}.

% Fractional time differentiability of D^2v
\begin{prop}
Assume $\bv$ is a weak solution of \eqref{mainPDE} in $U\times (0,T)$ and that $F$ satisfies \eqref{DtwoFCalpha}.  For each open $V\subset\subset U$ and $[t_0,t_1]\in (0,T)$, there is a constant $C$ such that 
\begin{equation}\label{FractionD2VDeriv}
\int^{t_1}_{t_0}\int_V|D^2\bv(x,t+h)-D^2\bv(x,t)|^2dxdt\le C |h|^{\frac{\alpha}{2}}
\end{equation} 
for $0<|h|<\frac{1}{2}\min\{1,t_0,T-t_1\}$. 
\end{prop}

\begin{proof}
First let $h\in \left(0,\frac{1}{2}\min\{1,t_0,T-t_1\}\right)$ and assume $u\in C^\infty_c(U)$ with $0\le u\le 1$ and $u\equiv 1$ in $V$. In the computations below, we will 
omit the spatial variable of $\bv$ and its derivatives.

\par 1. Suppose $i\in \{1,\dots, n\}$. By the uniform convexity of $M\mapsto \frac{1}{2}D^2F(D\bv(t))(M, M)$ for each $t\in (0,T)$,
\begin{align*}
&\int_U u\left[\frac{1}{2}D^2F(D\bv(t))(D\bv_{x_i}(t+h), D\bv_{x_i}(t+h))-\frac{1}{2}D^2F(D\bv(t))(D\bv_{x_i}(t), D\bv_{x_i}(t))\right]dx\\
&\hspace{.5in}\ge \int_U uD^2F(D\bv(t))(D\bv_{x_i}(t), D\bv_{x_i}(t+h)-D\bv_{x_i}(t))dx \\
&\hspace{1in} + \frac{\lambda}{2}\int_U u|D\bv_{x_i}(t+h)-D\bv_{x_i}(t)|^2dx\\
&\hspace{.5in}\ge - \int_U \Div(uD^2F(D\bv(t))D\bv_{x_i}(t))\cdot (\bv_{x_i}(t+h)-\bv_{x_i}(t))dx \\
&\hspace{1in} + \frac{\lambda}{2}\int_U u|D\bv_{x_i}(t+h)-D\bv_{x_i}(t)|^2dx\\
&\hspace{.5in}\ge - \int_U (D^2F(D\bv(t))D\bv_{x_i}(t)Du + u\partial_{x_i}D\psi(\bv_t(t)))\cdot (\bv_{x_i}(t+h)-\bv_{x_i}(t))dx \\
&\hspace{1in}+ \frac{\lambda}{2}\int_V |D\bv_{x_i}(t+h)-D\bv_{x_i}(t)|^2dx\\
&\hspace{.5in}\ge - \int_U (D^2F(D\bv(t))D\bv_{x_i}(t)Du + uD^2\psi(\bv_t(t))\bv_{x_it}(t))\cdot (\bv_{x_i}(t+h)-\bv_{x_i}(t))dx \\
&\hspace{1in}+ \frac{\lambda}{2}\int_V |D\bv_{x_i}(t+h)-D\bv_{x_i}(t)|^2dx.
\end{align*}
Also observe that the same inequality holds with $t$ and $t+h$ reversed.  Adding these inequalities together gives 
\begin{align*}
&\lambda\int_V|D\bv_{x_i}(t+h)-D\bv_{x_i}(t)|^2dx \\
&\quad \le   \int_U (D^2F(D\bv(t))D\bv_{x_i}(t)Du + uD^2\psi(\bv_t(t))\bv_{x_it}(t))\cdot (\bv_{x_i}(t+h)-\bv_{x_i}(t))dx\\
&\quad  + \int_U (D^2F(D\bv(t+h))D\bv_{x_i}(t+h)Du + uD^2\psi(\bv_t(t+h))\bv_{x_it}(t+h))\cdot (\bv_{x_i}(t)-\bv_{x_i}(t+h))dx\\
&\quad +\frac{1}{2}\int_Uu\left(D^2F(D\bv(t))-D^2F(D\bv(t+h))\right)(D\bv_{x_i}(t+h), D\bv_{x_i}(t+h))dx\\
&\quad +\frac{1}{2}\int_Uu\left(D^2F(D\bv(t+h))-D^2F(D\bv(t))\right)(D\bv_{x_i}(t), D\bv_{x_i}(t))dx.
\end{align*}

\par 2. Observe
\begin{align*}
&\int_U (D^2F(D\bv(t))D\bv_{x_i}(t)Du + uD^2\psi(\bv_t(t))\bv_{x_it}(t))\cdot (\bv_{x_i}(t+h)-\bv_{x_i}(t))dx\\
&\quad=\int_V (D^2F(D\bv(t))D\bv_{x_i}(t)Du + uD^2\psi(\bv_t(t))\bv_{x_it}(t))\cdot (\bv_{x_i}(t+h)-\bv_{x_i}(t))dx\\
&\quad \le \left(\int_V|D^2F(D\bv(t))D\bv_{x_i}(t)Du + uD^2\psi(\bv_t(t))\bv_{x_it}(t)|^2dx \right)^{1/2}\left(\int_V|\bv_{x_i}(t+h)-\bv_{x_i}(t)|^2dx\right)^{1/2}\\
&\quad \le \left(2\int_V\Lambda^2|Du|^2|D^2\bv(t)|^2 + \Theta^2u^2|D\bv_t(t)|^2dx \right)^{1/2}\left(\int_V|D\bv(t+h)-D\bv(t)|^2dx\right)^{1/2}\\
&\quad \le \left(2\int_V\Lambda^2|Du|^2|D^2\bv(t)|^2 + \Theta^2u^2|D\bv_t(t)|^2dx \right)^{1/2}\left(\int_V\int^{t+h}_t|D\bv_t(s)|^2dsdx\cdot h\right)^{1/2}\\
&\quad \le \sqrt{2}\left(\int_V\Lambda^2|Du|^2|D^2\bv(t)|^2 + \Theta^2u^2|D\bv_t(t)|^2dx \right)^{1/2}\left(\int^{(t_1+T)/2}_{t_0}\int_V|D\bv_t|^2dxds\right)^{1/2}h^{1/2}.
\end{align*}
As a result, 
\begin{align*}
&\left|\int^{t_1}_{t_0}\int_U (D^2F(D\bv(t))D\bv_{x_i}(t)Du + uD^2\psi(\bv_t(t))\bv_{x_it}(t))\cdot (\bv_{x_i}(t+h)-\bv_{x_i}(t))dxdt\right|\\
& \le (2h)^{1/2}\left(\int^{(t_1+T)/2}_{t_0}\int_V|D\bv_t|^2dxds\right)^{1/2}\int^{t_1}_{t_0}\left(\int_V\Lambda^2|Du|^2|D^2\bv(t)|^2 + \Theta^2u^2|D\bv_t(t)|^2dx \right)^{1/2}dt\\
& \le(2h)^{1/2}\left(\int^{(t_1+T)/2}_{t_0}\int_V|D\bv_t|^2dxds\right)^{1/2}\times \\
&\quad\quad\left(\int^{t_1}_{t_0}\int_V\Lambda^2|Du|^2|D^2\bv(t)|^2 + \Theta^2u^2|D\bv_t(t)|^2dxdt \right)^{1/2}(t_1-t_0)^{1/2}.
\end{align*}

Likewise, we find 
\begin{align*}
&\left|\int^{t_1}_{t_0}\int_U (D^2F(D\bv(t+h))D\bv_{x_i}(t+h)Du + uD^2\psi(\bv_t(t+h))\bv_{x_it}(t+h))\cdot (\bv_{x_i}(t)-\bv_{x_i}(t+h))dxdt\right|\\
& \le(2h)^{1/2}\left(\int^{(t_1+T)/2}_{t_0}\int_V|D\bv_t|^2dxds\right)^{1/2} \times \\
&\quad\quad \left(\int^{t_1}_{t_0}\int_V\Lambda^2|Du|^2|D^2\bv(t+h)|^2 + \Theta^2u^2|D\bv_t(t+h)|^2dxdt \right)^{1/2}(t_1-t_0)^{1/2}\\
& \le(2h)^{1/2}\left(\int^{(t_1+T)/2}_{t_0}\int_V|D\bv_t|^2dxds\right)^{1/2} \times \\
&\quad\quad \left(\int^{(t_1+T)/2}_{t_0}\int_V\Lambda^2|Du|^2|D^2\bv(s)|^2 + \Theta^2u^2|D\bv_t(s)|^2dxds \right)^{1/2}(t_1-t_0)^{1/2}.
\end{align*}

\par 3. Recall that we may assume \eqref{alphaExtra}. With this assumption, we can again apply H\"older's inequality to get
\begin{align*}
&\int_Uu\left(D^2F(D\bv(t))-D^2F(D\bv(t+h))\right)\left(D\bv_{x_i}(t+h), D\bv_{x_i}(t+h)\right)dx\\
&\quad \le \int_Uu\left|D^2F(D\bv(t))-D^2F(D\bv(t+h))\right| |D\bv_{x_i}(t+h)|^2dx\\
&\quad \le \left(\int_Uu\left|D^2F(D\bv(t))-D^2F(D\bv(t+h))\right|^{\frac{p}{p-2}}dx\right)^{1-2/p}\left(\int_Uu|D\bv_{x_i}(t+h)|^pdx\right)^{2/p}\\
&\quad \le C\left(\int_Uu\left|D\bv(t)-D\bv(t+h)\right|^{\alpha\left(\frac{p}{p-2}\right)}dx\right)^{1-2/p}\left(\int_Uu|D\bv_{x_i}(t+h)|^pdx\right)^{2/p}\\
&\quad \le C\left(\int_Uu\left|D\bv(t)-D\bv(t+h)\right|^{2}dx\right)^{\alpha/2}|V|^{\frac{p-2}{p}-\frac{\alpha}{2}}\left(\int_Uu|D\bv_{x_i}(t+h)|^pdx\right)^{2/p}\\
&\quad \le C\left(\int^{t+h}_{t}\int_Uu\left|D\bv_t\right|^{2}dxds\right)^{\alpha/2}h^{\frac{\alpha}{2}}|V|^{\frac{p-2}{p}-\frac{\alpha}{2}}\left(\int_Uu|D^2\bv(t+h)|^pdx\right)^{2/p}\\
&\quad \le C\left(\int^{(t_1+T)/2}_{t_0}\int_Uu\left|D\bv_t\right|^{2}dxds\right)^{\alpha/2}h^{\frac{\alpha}{2}}|V|^{\frac{p-2}{p}-\frac{\alpha}{2}}\left(\int_Uu|D^2\bv(t+h)|^pdx\right)^{2/p}.
\end{align*}
In particular, 
\begin{align*}
&\int^{t_1}_{t_0}\int_Uu\left(D^2F(D\bv(t))-D^2F(D\bv(t+h))\right)D\bv_{x_i}(t+h)\cdot D\bv_{x_i}(t+h)dxdt\\
&\quad \le C\left(\int^{(t_1+T)/2}_{t_0}\int_Uu\left|D\bv_t\right|^{2}dxds\right)^{\alpha/2}h^{\frac{\alpha}{2}}|V|^{\frac{p-2}{p}-\frac{\alpha}{2}}\int^{t_1}_{t_0}\left(\int_Uu|D^2\bv(t+h)|^pdx\right)^{2/p}dt\\
&\quad \le C\left(\int^{(t_1+T)/2}_{t_0}\int_Uu\left|D\bv_t\right|^{2}dxds\right)^{\alpha/2}h^{\frac{\alpha}{2}}|V|^{\frac{p-2}{p}-\frac{\alpha}{2}}\left(\int^{t_1}_{t_0}\int_Uu|D^2\bv(t+h)|^pdxdt\right)^{2/p}(t_1-t_0)^{1-2/p}\\
&\quad \le C\left(\int^{(t_1+T)/2}_{t_0}\int_Uu\left|D\bv_t\right|^{2}dxds\right)^{\alpha/2}h^{\frac{\alpha}{2}}|V|^{\frac{p-2}{p}-\frac{\alpha}{2}}\left(\int^{(t_1+T)/2}_{t_0}\int_Uu|D^2\bv(s)|^pdxds\right)^{2/p}(t_1-t_0)^{1-2/p}.
\end{align*}
Analogously, 
\begin{align*}
&\int^{t_1}_{t_0}\int_Uu\left(D^2F(D\bv(t+h))-D^2F(D\bv(t))\right)D\bv_{x_i}(t)\cdot D\bv_{x_i}(t)dxds\\
&\quad \le C\left(\int^{(t_1+T)/2}_{t_0}\int_Uu\left|D\bv_t\right|^{2}dxds\right)^{\alpha/2}h^{\frac{\alpha}{2}}|V|^{\frac{p-2}{p}-\frac{\alpha}{2}}\left(\int^{t_1}_{t_0}\int_Uu|D^2\bv|^pdxds\right)^{2/p}(t_1-t_0)^{1-2/p}.
\end{align*}
\par 4. Putting all of these estimates  together we find a constant $C_+$ independent of $h$ such that 
$$
\int^{t_1}_{t_0}\int_V|D^2\bv(x,t+h)-D^2\bv(x,t)|^2dx\le C_+ h^{\frac{\alpha}{2}}.
$$
This bound clearly implies \eqref{FractionD2VDeriv} for $h\in \left(0,\frac{1}{2}\min\{1,t_0,T-t_1\}\right)$. 
It is also not difficult to see how to use the ideas above to justify \eqref{FractionD2VDeriv} for $h\in \left(-\frac{1}{2}\min\{1,t_0,T-t_1\},0\right)$. We leave the details to the reader. 
\end{proof}

A direct consequence of the preceding proposition is that $D^2\bv$ is fractionally differentiable in time as exhibited in \eqref{FractionD2VDeriv2} below.  

\begin{cor}\label{SecondDiffDVCor}
Assume $\bv$ is a weak solution of \eqref{mainPDE} in $U\times (0,T)$ and $F$ satisfies \eqref{DtwoFCalpha} for some $\alpha\in (0,1]$.  For each open $V\subset\subset U$, $[t_0,t_1]\in (0,T)$, and
\begin{equation}\label{BetaInterval2}
\beta\in \left(0,\frac{\alpha}{2}\right),
\end{equation}
there is constant $A=A(p,\alpha,\beta,t_0,t_1, V)>0$ such that 
\begin{equation}\label{FractionD2VDeriv2}
\int^{t_1}_{t_0}\int^{t_1}_{t_0}\int_V\frac{|D^2\bv(x,t)-D^2\bv(x,s)|^2}{|t-s|^{1+\beta}}dxdtds\le A\left(C+\|D^2\bv\|^2_{L^2(V\times[t_0,t_1])}\right).
\end{equation} 
Here $C$ is the constant in \eqref{FractionVtDeriv}.
\end{cor}

\par It turns out that we can use these fractional time derivative estimates to investigate compactness properties of weak solutions. To this end, we will  
make use of the following compactness theorem due to J. Simon.

\begin{thm*}\textup{(Theorem 1 of \cite{Simon})}  Let $X$ be a Banach space over $\R$ with norm $\|\cdot \|$ and $p\in [1,\infty)$. Suppose $\{f^k\}_{k\in \N}\subset L^p([t_0,t_1]; X)$ with 
$$
\left\{\int^{s_1}_{s_0}f^k(t)dt\right\}_{k\in \N}
$$
relatively compact in $X$ for all $t_0< s_0\le s_1< t_1$ and 
$$
\lim_{h\rightarrow 0^+}\sup_{k\in \N}\int^{t_1-h}_{t_0}\|f^k(t+h)-f^k(t)\|^pdt=0.
$$
Then there is a subsequence $\{f^{k_j}\}_{j\in \N}$  and $f\in L^p([t_0,t_1]; X)$ such that 
$f^{k_j}\rightarrow f$ in $L^p([t_0,t_1]; X)$.  
\end{thm*}
\begin{rem}\label{AASimon} Theorem 1 of \cite{Simon} also asserts the following generalization of the Arzel\'a-Ascoli criterion.  That is, suppose $\{f^k\}_{k\in \N}\subset C([t_0,t_1]; X)$ with 
$$
\left\{\int^{s_1}_{s_0}f^k(t)dt\right\}_{k\in \N}
$$
relatively compact in $X$ for all $t_0< s_0\le s_1< t_1$ and
$$
\lim_{h\rightarrow 0^+}\sup_{k\in \N}\left\{\max_{t_0\le t\le t_1-h}\|f^k(t+h)-f^k(t)\|\right\}=0.
$$
Then there is a subsequence $\{f^{k_j}\}_{j\in \N}$  and $f\in C([t_0,t_1]; X)$ such that 
$f^{k_j}\rightarrow f$ in $C([t_0,t_1]; X)$.
\end{rem}
Our central compactness result is as follows. 
\begin{prop}\label{CompactnessProp}
Assume $\psi^k\in C^2(\R^m)$ and $F^k\in C^2(\Mmn)$ satisfy  \eqref{UnifConv}, \eqref{UnifConv2} and \eqref{psiFzeroAtzero} for each $k\in \N$. Further suppose $\{\bv^k\}_{k\in \N}$ is a sequence of weak solutions of 
\be
D\psi^k(\bv^k_t)=\text{\normalfont{div}}DF^k(D\bv^k)
\ee
in $U\times (0,T)$ such that
\be\label{StartingBound}
\sup_{k\in \N}\int^T_0\int_U(|\bv^k|^2+|D\bv^k|^2+|\bv_t^k|^2)dxdt<\infty.
\ee
Then there is $\psi\in C^1(\R^m)$ and $F\in C^1(\Mmn)$ satisfying \eqref{UnifConv}, \eqref{UnifConv2} and \eqref{psiFzeroAtzero}, a subsequence $\{\bv^{k_j}\}_{j\in \N}$, and a weak solution $\bv$ of 
\be
D\psi(\bv_t)=\text{\normalfont{div}} DF(D\bv)
\ee
in $U\times(0,T)$ such that for each $[t_0,t_1]\subset (0,T)$ and open $V\subset\subset U$
\be
\bv^{k_j}\rightarrow \bv \; \text{in}\;
\begin{cases}
C([t_0,t_1]; H^1(V;\R^m))\\
L^2([t_0,t_1]; H^2(V;\R^m))
\end{cases}
\ee
and 
\be
\bv^{k_j}_t\rightarrow \bv_t \; \text{in}\; L^2([t_0,t_1];L^2(V;\R^m)).
\ee
\end{prop}

\begin{proof}
1. By assumption, we have 
$$
|D\psi^k(z_1)-D\psi^k(z_2)|\le \Theta|z_1-z_2|\quad (z_1,z_2\in \R^m)
$$
for each $k\in \N$. Since $|D\psi^k(0)|=0$ (by \eqref{psiFzeroAtzero}), the sequence $(D\psi^k)_{k\in \N}$ is both equicontinuous and locally uniformly bounded on $\R^m$.  By the Arzel\`a-Ascoli Theorem, there is a subsequence  $(\psi^{k_j})_{j\in \N}$ and $\psi \in C^{1}(\R^m)$ such that $\psi^{k_j} \rightarrow \psi$ and $D\psi^{k_j} \rightarrow D\psi$ locally uniformly on $\R^m$.  Moreover, $\psi$ satisfies \eqref{UnifConv}.  Analogously, there is a subsequence $(F^{k_j})_{j\in \N}$ and $F\in C^1(\Mmn)$ for which $F^{k_j} \rightarrow F$ and $DF^{k_j} \rightarrow DF$ locally uniformly on $\Mmn$ and $F$ satisfies \eqref{UnifConv2}, as well. Clearly, $\psi$ and $F$ additionally satisfy \eqref{psiFzeroAtzero}.

\par We also have by \eqref{StartingBound} and Rellich compactness that there is $\bv\in H^1(U\times (0,T);\R^m)$ and a subsequence  $\{\bv^{k_j}\}_{j\in \N}$ such that 
$$
\begin{cases}
\bv^{k_j}\rightarrow \bv &\; \text{in}\; L^2(U\times(0,T);\R^m)\\
D\bv^{k_j}\rightharpoonup D\bv&\; \text{in}\; L^2(U\times(0,T);\Mmn )\\
\bv^{k_j}_t\rightharpoonup \bv_t&\; \text{in}\; L^2(U\times(0,T);\R^m).
\end{cases}
$$
We will now proceed to strengthen these convergence assertions and then show that $\bv$ is a weak solution as claimed.  Each convergence assertion will follow 
from Simon's theorem. 

\par 2. We will first argue that $\bv^k:[0,T]\rightarrow L^2(U;\R^m)$ converges uniformly.  Suppose $(s_0,s_1)\subset [0,T]$, and set
$$
\bw^k(x):=\int^{s_1}_{s_0}\bv^k(x,t)dt\quad (x\in U)
$$
for $k\in \N$.  In view of \eqref{StartingBound}, $\{\bw^k\}_{k\in \N}$ is bounded in $H^1(U;\R^m)$ and is thus precompact in $L^2(U;\R^m)$.  Note also that for $h$ sufficiently small,
\begin{align*}
\int_U|\bv^k(x,t+h)-\bv^k(x,t)|^2dx&=\int_U\left|\int^{t+h}_t\bv^k_t(x,s)ds\right|^2dx\\
&\le h\int^{t+h}_t\int_U|\bv^k_t|^2dxds\\
&\le h\left(\int^{T}_{0}\int_U|\bv^k_t|^2dxds\right).
\end{align*}
Thus 
$$
\lim_{h\rightarrow 0}\sup_{k\in \N}\left\{\max_{0\le t\le T-h}\int_U|\bv^k(x,t+h)-\bv^k(x,t)|^2dx\right\}=0.
$$
By Simon's theorem, there is a subsequence  $\{\bv^{k_j}\}_{j\in \N}$ converging to $\bv$ in $C([0,T]; L^2(U;\R^m))$. 

\par 3. Let us now argue that a subsequence of $D\bv^k:[t_0,t_1]\rightarrow L^2(V;\Mmn)$ converges uniformly. Let $(s_0,s_1)\subset [t_0,t_1]$ and define
$$
{\bf \xi}^k(x):=\int^{s_1}_{s_0}D\bv^k(x,t)dt\quad (x\in V)
$$
for $k\in \N$.  Since $\partial U$ is smooth, we may assume without loss of generality that $\partial V$ is smooth; or else we can select an open $W$ with $V\subset\subset W\subset\subset U$ and $\partial W$ smooth and verify $D\bv^k:[t_0,t_1]\rightarrow L^2(W;\Mmn)$ converges uniformly. With this assumption and the estimate \eqref{H2EstimateSpec}, we
have that $\{\xi^k\}_{k\in \N}$ is bounded in $H^1(V;\Mmn)$ and is thus precompact in $L^2(V;\Mmn)$.

\par We also have 
\be\label{DvtCompactness}
\int^{t_1}_{t_0}\int_V|D\bv^k_t|^2dxds\le C \quad (k\in \N),
\ee
for some $C$, by inequality \eqref{EnergyBound2}. Therefore, for $h$ sufficiently small and $t\in [t_0,t_1-h]$,
\begin{align*}
\int_V|D\bv^k(x,t+h)-D\bv^k(x,t)|^2dx&\le h\int^{t+h}_t\int_V|D\bv^k_t|^2dxds\\
&\le h\left(\int^{t_1}_{t_0}\int_V|D\bv^k_t|^2dxds\right)\\
&\le Ch.
\end{align*}
In view of Simon's theorem, there is a subsequence $\{D\bv^{k_j}\}_{j\in \N}$ converging to $D\bv$ in \\ $C([0,T]; L^2(V;\Mmn))$. So we conclude that there is a subsequence (not relabeled) for which $\bv^{k_j}\rightarrow \bv$ in $C([t_0,t_1]; H^1(V;\R^m))$.

\par 4. We can also prove that (a subsequence of) $\bv^{k_j}_t$ converges to $\bv_t$ in $L^2([t_0,t_1]; L^2(V;\R^m))$ using similar computations as above. Indeed, in view of \eqref{SpecificVtFractEstimate}, we can find a constant $C$ independent of $k\in \N$ such that 
$$
%\sup_{t_0\le t\le t_1-h}\int_V|\bv^k_t(x,t+h)-\bv^k_t(x,t)|^2dxdt\le C h^{\frac{1}{2}-\frac{1}{p}}
\int_{t_0}^{t_1-h}\int_V|\bv^k_t(x,t+h)-\bv^k_t(x,t)|^2dxdt\le C h^{\frac{1}{2}-\frac{1}{p}}
$$
for all $h>0$ small enough.  Moreover, for each $(s_0,s_1)\subset [t_0,t_1]$, the sequence of functions 
$$
\bw^k(x):=\int^{s_1}_{s_0}\bv^k_t(x,t)dt
$$
is bounded in $H^1(V;\R^m)$ by \eqref{StartingBound} and \eqref{DvtCompactness}. So  $\bv^{k_j}_t$ converges to $\bv_t$ in $L^2([t_0,t_1]; L^2(V;\R^m))$ by Simon's theorem.  In an analogous fashion, we can use the $H^3_{\text{loc}}$ estimate \eqref{H3EstimateYeah} and the fractional time derivative bound \eqref{FractionD2VDeriv} to conclude that $D^2\bv^{k_j}$ converges to $D^2\bv$ in $L^2([t_0,t_1]; L^2(V; {\cal S}^2(\R^n;\R^m)))$. 
We leave the details to the reader. 

\par 5. Finally, we need to argue that $\bv$ is a weak solution of \eqref{mainPDE}. To this end, it suffices to show that \eqref{WeakSolnCond} holds. Of course, we have 
\be\label{PsikayJay}
\int_U D\psi^{k_j}(\bv^{k_j}_t(x,t))\cdot  \bw(x)dx + \int_U DF^{k_j}(D\bv^{k_j}(x,t))\cdot D\bw(x)dx=0
\ee
for each $j\in \N$, $t\in (0,T)$ and  $\bw\in H^1_0(U;\R^m)$.  Passing to a further subsequence if necessary, we may assume that $\bv_t^{k_j}(\cdot, t)$ converges to $\bv_t(\cdot, t)$ almost everywhere in $U$ for almost every $t\in (0,T)$. By the local uniform convergence of $D\psi^{k_j}$ to $D\psi$, we have $D\psi^{k_j}(\bv_t^{k_j}(\cdot, t))\rightarrow D\psi(\bv_t(\cdot, t))$ almost everywhere in $U$ for almost every $t\in (0,T)$. 

\par Since 
$$
|D\psi^{k_j}(\bv_t^{k_j}(\cdot, t))|\le\Theta|\bv_t^{k_j}(\cdot, t)|,
$$
we can apply a standard variant of Lebesgue's dominated convergence theorem (Theorem 4, section 1.3 of \cite{EvaGar}) to deduce 
$$
\lim_{j\rightarrow\infty}\int_U D\psi^{k_j}(\bv^{k_j}_t(x,t))\cdot  \bw(x)dx =\int_U D\psi(\bv_t(x,t))\cdot  \bw(x)dx
$$
for almost every $t\in (0,T)$. Likewise, we may we conclude 
$$
\lim_{j\rightarrow\infty}\int_U DF^{k_j}(D\bv^{k_j}(x,t))\cdot D\bw(x)dx =\int_U DF(D\bv(x,t))\cdot D\bw(x)dx
$$
for every $t\in (0,T)$. Therefore, we may pass to the limit as $j\rightarrow\infty$ in \eqref{PsikayJay} and conclude that $\bv$ is indeed a weak solution of \eqref{mainPDE}. 
\end{proof}

%%%%%%%%%%%%%%%%%%%%%%%%%%%%%%%%%%%%%%%%%%%%%%%%%%%%%%%%%
\section{Partial regularity}
We now proceed to proving Theorem \ref{mainThm}.  Consequently, we will assume throughout 
this section that $D^2F$ is H\"older continuous with exponent $\alpha\in (0,1]$ as in \eqref{DtwoFCalpha}. We will first use Proposition
\ref{CompactnessProp} to verify a decay property of a quantity that measures the local energy of weak solutions. Then we will iterate 
this decay property to derive a criterion for local H\"older continuity of weak solutions.   Our final task will be to estimate the parabolic 
Hausdorff dimension (Definition \ref{ParaHausMeas} below) of the set of points where this  criterion for local H\"older continuity may fail. 
%The key to this approach will be the fractional time derivative estimates \eqref{FractionVtDeriv} and \eqref{FractionD2VDeriv}. 

 \par We will denote a parabolic cylinder of radius $r>0$ centered at $(x,t)$ as
$$
Q_r(x,t):=B_r(x)\times(t-r^2/2, t+r^2/2)
$$
and the average of a mapping $\bw$ over $Q_r=Q_r(x,t)$ as
$$
\bw_{Q_r}=\ffint\;\bw:=\frac{1}{|Q_r|}\iint_{Q_r}\bw(y,s)dyds.
$$
For a given weak solution $\bv$, quantity that will be of great utility to us is the local space-time energy
\begin{align}\label{localEE}
E(x,t,r)&:=\ffint|\bv_t - (\bv_t)_{Q_r}|^2 dyds+ \ffint\;\; \left|\frac{D\bv - (D\bv)_{Q_r}- (D^2\bv)_{Q_r}(y-x) }{r}\right|^2dyds \nonumber \\
&\hspace{1.5in}+ \ffint|D^2\bv - (D^2\bv)_{Q_r}|^2 dyds,
\end{align}
which is defined for $Q_r=Q_r(x,t)\subset U\times(0,T)$  and $r>0$.  Here,
$y\mapsto (D^2\bv)_{Q_r}(y-x)$ is the $m\times n$ matrix valued mapping with $i,j$th component function $y\mapsto (Dv^i_{x_j})_{Q_r}\cdot (y-x)$.

\par An important decay property of $E$ is as follows.    
\begin{lem}\label{BlowUplemma}
Let $L>0$ and $\gamma\in \left(0,\alpha\right)$.  There are $\epsilon, \rho, \vartheta\in \left(0,\left(\frac{1}{2}\right)^{1/\alpha}\right)$ such that if
\be
\begin{cases}
Q_r:=Q_r(x,t)\subset U\times(0,T),\; r<\rho\\\\
|(\bv_t)_{Q_r}|, |(D\bv)_{Q_r}|, |(D^2\bv)_{Q_r}|\le L\\\\
E(x,t,r)\le \epsilon^2,
\end{cases}
\ee
then 
\be
E(x,t,\vartheta r)\le \frac{1}{2}E(x,t,r)\quad \text{\normalfont{or}}\quad E(x,t,r)\le Lr^{2\gamma}.
\ee
\end{lem}

\begin{proof}
1. If not, there are $L_0>0$, $\gamma_0\in (0,\alpha)$ and sequences $(x_k,t_k)\in U\times(0,T)$, $r_k\rightarrow 0$,  $\epsilon_k\rightarrow 0$, $\vartheta_k\equiv \vartheta\in \left(0,\left(\frac{1}{2}\right)^{1/\alpha}\right)$ (chosen below) such that 
\be
\begin{cases}
Q_{r_k}:=Q_{r_k}(x_k,t_k)\subset U\times(0,T)\\\\
|(\bv_t)_{Q_{r_k}}|, |(D\bv)_{Q_{r_k}}|, |(D^2\bv)_{Q_{r_k}}|\le L_0\\\\
E(x_k,t_k,r_k)=\epsilon_k^2,
\end{cases}
\ee
while 
\be\label{blowupAlternative}
E(x_k,t_k,\vartheta_k r_k)> \frac{1}{2}\epsilon_k^2\quad\text{and}\quad \epsilon_k^2> L_0r_k^{2\gamma_0}.
\ee

\par For each $k\in\N$ and $(y,s)\in Q_1:=Q_1(0,0)$, define 
$$
\tilde\bv^k(y,s):=\frac{\bv(x_k+r_ky, t_k+r_k^2s)-(\bv)_{Q_{r_k}} -(\bv_t)_{Q_{r_k}}r_k^2s -(D\bv)_{Q_{r_k}}r_ky -\frac{1}{2}(D^2\bv)_{Q_{r_k}}r^2_k(y, y) }{\epsilon_kr_k^2},
$$
and 
$$
\bv^k(y,s):=\tilde\bv^k(y,s)-(\tilde\bv^k)_{Q_1}.
$$
Note that since $E(x_k,t_k,r_k)=\epsilon_k^2$, 
\be\label{blowupIntegralBound}
\fffint(|\bv^k_s|^2+|D\bv^k|^2+|D^2\bv^k|^2)dyds =1.
\ee
As $(\bv^k)_{Q_1}=0$, \eqref{blowupIntegralBound} implies a uniform bound on $\iint_{Q_1}|\bv^k|^2dyds$ by 
 Poincar\'e's inequality for $H^1(Q_1;\R^m)$ mappings with zero average. 

\par 2. Direct computation also shows that 
$$
D\psi(a_k+\epsilon_k \bv^k_s)=\frac{1}{r_k}\Div DF(\xi_k+r_k N_k y +\epsilon_k r_k D\bv^k)
$$
weakly in $Q_1$. The sequences 
$$
a_k:=(\bv_t)_{Q_{r_k}}\quad \xi_k:=(D\bv)_{Q_{r_k}}\quad N_k:=(D^2\bv)_{Q_{r_k}}
$$
are all bounded, so without loss of generality we may assume that $a_k\rightarrow a\in \R^m$, $\xi_k\rightarrow \xi \in \Mmn$ and $N_k\rightarrow N\in {\cal S}^2(\R^n,\R^m)$.   We may also write 
\begin{align*}
\frac{D\psi(a_k+\epsilon_k \bv^k_s)-D\psi(a_k)}{\epsilon_k}&=\Div\left[\frac{DF(\xi_k+r_k N_k y +\epsilon_k r_k D\bv^k)-DF(\xi_k+r_k N_k y)}{\epsilon_kr_k}\right]\\
&\quad\quad + \quad \frac{D^2F(\xi_k+r_k N_k y)\cdot N_k-D\psi(a_k)}{\epsilon_k}.
\end{align*}
Here $y\mapsto D^2F(\xi_k+r_k N_k y)\cdot N_k$ is the $\R^m$ valued mapping with $i$th component function
$$
y\mapsto \sum^m_{r=1}\sum^n_{j,\ell=1}F_{M^i_jM^r_\ell}(\xi_k+r_k N_k y)(v^r_{x_j x_\ell})_{Q_{r_k}}
$$
$(i=1,\dots, m)$.

\par Therefore,  
\be\label{psikFkSystem}
D\psi^k(\bv^k_s)=\Div\left[D_MF^k(D\bv^k,y)\right]+f_k(y),
\ee
where 
$$
\begin{cases}
\psi^k(w):=\displaystyle\frac{\psi(a_k+\epsilon_k w)-\psi(a_k)-D\psi(a_k)\cdot \epsilon_k w}{\epsilon^2_k}\\\\
F^k(M,y):=\displaystyle\frac{F(\xi_k+r_kN_k y+\epsilon_kr_k M)-F(\xi_k+r_kN_k y)-DF(\xi_k+r_kN_k y)\cdot \epsilon_kr_k M}{(\epsilon_kr_k)^2}\\\\
f^k(y):=\displaystyle\frac{D^2F(\xi_k+r_k N_k y)\cdot N_k-D\psi(a_k)}{\epsilon_k}.
\end{cases}
$$

\par 3. Combining the H\"older continuity of $D^2F$ \eqref{DtwoFCalpha} and \eqref{blowupAlternative},  
\begin{align*}
|f^k(y_1)-f^k(y_2)|&=\frac{1}{\epsilon_k}\left|(D^2F(\xi_k+r_k N_k y_1)-D^2F(\xi_k+r_k N_k y_2))\cdot N_k\right|\\
&\le \frac{|N_k|}{\epsilon_k}\left|D^2F(\xi_k+r_k N_k y_1)-D^2F(\xi_k+r_k N_k y_2)\right|\\
&\le \frac{|N_k|^{1+\alpha}r_k^\alpha}{\epsilon_k}\left| y_1-y_2\right|^\alpha\\
&\le \frac{|N_k|^{1+\alpha}r_k^{\alpha-\gamma_0}}{\sqrt{L_0}}\left| y_1-y_2\right|^\alpha
\end{align*}
for each $y_1,y_2\in B_1$. Therefore, $f^k\in C^\alpha(\overline{B_1};\R^m)$ is uniformly equicontinuous.  In view of \eqref{psikFkSystem}, we can choose $\phi\in C^\infty_c(Q_1)$ be nonnegative with $\iint_{Q_1}\phi dyds=1$ to get 
\begin{align*}
f_k(0)&=\iint_{Q_1}f_k(0)\phi dyds\\
&=\iint_{Q_1}(f_k-f_k(0))\phi dyds+\iint_{Q_1}f_k\phi dyds\\
&=O(r_k^{\alpha-\gamma_0})+\iint_{Q_1}\frac{D\psi(a_k+\epsilon_k \bv^ks)-D\psi(a_k)}{\epsilon_k}\phi dyds  \\
&\hspace{.8in}+\iint_{Q_1}\frac{DF(\xi_k+r_k N_k y +\epsilon_k r_k D\bv^k)-DF(\xi_k+r_k N_k y)}{\epsilon_kr_k}\cdot D\phi dyds.
\end{align*}
\par We can now employ \eqref{blowupIntegralBound} to find
\begin{align*}
|f_k(0)|&\le O(r_k^{\alpha-\gamma_0})+\Theta \iint_{Q_1}|\bv^k_s||\phi| dyds + \Lambda \iint_{Q_1}|D\bv^k||D\phi| dyds\\
&\le O(r_k^{\alpha-\gamma_0})+\max\{\Theta,\Lambda\}\left(\iint_{Q_1}(|\bv^k_s|^2+|D\bv^k|^2)dyds\right)^{1/2}\left(\iint_{Q_1}(|\phi|^2+|D\phi|^2)dyds\right)^{1/2}\\
&\le O(r_k^{\alpha-\gamma_0})+\max\{\Theta,\Lambda\}|Q_1|^{1/2}\left(\iint_{Q_1}(|\phi|^2+|D\phi|^2)dyds\right)^{1/2}.
\end{align*}
Consequently, $f_k$ is also uniformly pointwise bounded and thus converges (up to a subsequence) locally uniformly to a fixed vector $\beta\in \R^m$ which 
satisfies 
\be\label{betaEst}
|\beta|\le \max\{\Theta,\Lambda\}|Q_1|^{1/2}\left(\iint_{Q_1}(|\phi|^2+|D\phi|^2)dyds\right)^{1/2}.
\ee

\par Observe that $\psi^k$ and $F^k(\cdot,y)$ satisfy the hypotheses of Proposition \ref{CompactnessProp} for each $y\in B_1$. Using the same ideas to prove this proposition, we can conclude that there is a subsequence $(\bv^{k_j})_{j\in \N}$ and mapping $\bw$ such that  
\be
\bv^{k_j}\rightarrow \bw \; \text{in}\;
\begin{cases}
C\left(\left[-\frac{1}{2}R^2,\frac{1}{2}R^2\right]; H^1(B_R;\R^m)\right)\\
L^2\left(\left[-\frac{1}{2}R^2,\frac{1}{2}R^2\right]; H^2(B_R;\R^m)\right)
\end{cases}
\ee
and 
\be
\bv^{k_j}_t\rightarrow \bw_t \; \text{in}\; L^2\left(\left[-(1/2)R^2,(1/2)R^2\right]; L^2(B_R;\R^m)\right)
\ee
for each $R\in (0,1)$. Moreover, $\bw$ is a weak solution of the linear PDE 
$$
D^2\psi(a)\bw_s=\Div(D^2F(\xi)D\bw)+\beta
$$
in $Q_1$. 

\par 4. Using \eqref{blowupIntegralBound} and \eqref{betaEst}, it can be shown that there is a constant $C_1=C_1(m,n,\theta,\Theta,\lambda,\Lambda)$ such that 
\begin{align*}
\fthetaint|\bw_s-(\bw_s)_{Q_{\vartheta}}|^2dyds+\fthetaint\;\;\left|\frac{D\bw-(D\bw)_{Q_\vartheta}-(D^2\bw)_{Q_\vartheta}y}{\vartheta}\right|^2dyds\\ +\fthetaint|D^2\bw-(D^2\bw)_{Q_{\vartheta}}|^2dyds \le C_1\vartheta^2
\end{align*}
for every $\vartheta\in (0,1/2)$; see part 3 of the proof of Lemma 4.1 in \cite{Hynd} for a detailed verification of this fact. We now select $\vartheta_k\equiv\vartheta\in \left(0,\left(\frac{1}{2}\right)^{1/\alpha}\right)$ so small that $C_1\vartheta^2\le 1/4.$ The strong convergence of $\bv^{k_j}\rightarrow \bw$ alluded to above gives
\begin{align*}
\fthetaint|\bv^{k_j}_s-(\bv^{k_j}_s)_{Q_{\vartheta}}|^2dyds+\fthetaint\;\;\left|\frac{D\bv^{k_j}-(D\bv^{k_j})_{Q_\vartheta}-(D^2\bv^{k_j})_{Q_\vartheta}y}{\vartheta}\right|^2dyds\\ +\fthetaint|D^2\bv^{k_j}-(D^2\bv^{k_j})_{Q_{\vartheta}}|^2dyds \le \frac{3}{8}
\end{align*}
for all $j\in \N$ sufficiently large. It is routine to check that this inequality is equivalent to 
$$
E(x_{k_j},t_{k_j},\vartheta_{k_j}r_{k_j})<\frac{3}{8}\epsilon^2_{k_j}
$$
for all $j\in \N$ sufficiently large, which contradicts \eqref{blowupAlternative} for $k=k_j$. 
\end{proof}

\par Let us also recall a basic fact about the general decay of local energy. 
\begin{lem}\label{backwardsDecayLem}
Let $\bv$ be a weak solution in $U\times (0,T)$ and $Q_r(x,t)\subset U\times (0,T)$. Then
\be\label{backwardsDecayLemEst}
E(x,t,\vartheta r)\le \frac{12}{\vartheta^{2(n+3)}}E(x,t,r)
\ee
for each $\vartheta\in (0,1]$.
\end{lem}
\begin{proof}
We will derive two basic inequalities and apply them to $E(x,t,\vartheta r)$. For convenience, we will write $Q_r$ for $Q_r(x,t)$ and $Q_{\vartheta r}$ for $Q_{\vartheta r}(x,t)$. 
\par 1. Let $f\in L^2_{\text{loc}}((0,T); H^1_{\text{loc}}(U))$.  Note 
\begin{align}\label{FirstGeneralDecay}
\left(\fthetarint |f-(f)_{Q_{\vartheta r}}|^2dyds\right)^{1/2}&\le\left(\fthetarint |f-(f)_{Q_{r}}|^2dyds\right)^{1/2}+|(f)_{Q_{r}}-(f)_{Q_{\vartheta r}}|\nonumber \\
&\le \frac{1}{\vartheta^{n/2+1}}\left(\ffint |f-(f)_{Q_{r}}|^2dyds\right)^{1/2} +\left|\fthetarint (f-(f)_{Q_{r}})dyds\right|\nonumber \\
&\le \frac{1}{\vartheta^{n/2+1}}\left(\ffint |f-(f)_{Q_{r}}|^2dyds\right)^{1/2} +\left(\fthetarint |f-(f)_{Q_{r}}|^2dyds\right)^{1/2} \nonumber\\
&\le  \frac{2}{\vartheta^{n/2+1}}\left(\ffint |f-(f)_{Q_{r}}|^2dyds\right)^{1/2}.
\end{align}
Consequently, 
\be\label{backwardsDecay1}
\fthetarint |f-(f)_{Q_{\vartheta r}}|^2dyds\le \frac{4}{\vartheta^{n+2}}\ffint |f-(f)_{Q_{r}}|^2dyds.
\ee

\par 2. Observe that the function $ (y,s)\mapsto f(y,s)-(Df)_{Q_{\vartheta r}}(y-x)$ has the same average of $f$ on any cylinder centered at $(x,t)$. By \eqref{FirstGeneralDecay},
\begin{align*}
&\left(\fthetarint |f-(f)_{Q_{\vartheta r}}-(Df)_{Q_{\vartheta r}}(y-x)|^2dyds\right)^{1/2}\\
&\hspace{1in}=\left(\fthetarint |[f-(Df)_{Q_{\vartheta r}}(y-x)]-(f)_{Q_{\vartheta r}}|^2dyds\right)^{1/2}\\
&\hspace{1in}\le\frac{2}{\vartheta^{n/2+1}}\left(\ffint |[f-(Df)_{Q_{\vartheta r}}(y-x)]-(f)_{Q_{r}}|^2dyds\right)^{1/2}\\
&\hspace{1in}\le\frac{2}{\vartheta^{n/2+1}}\left(\ffint |f-(f)_{Q_{r}}-(Df)_{Q_{\vartheta r}}(y-x)|^2dyds\right)^{1/2}\\
&\hspace{1in}\le\frac{2}{\vartheta^{n/2+1}}\left(\ffint |f-(f)_{Q_{r}}-(Df)_{Q_r}(y-x)|^2dyds\right)^{1/2}  \\
&\hspace{1.5in} +\frac{2}{\vartheta^{n/2+1}}r|(Df)_{Q_r}-(Df)_{Q_{\vartheta r}}|\\
&\hspace{1in}\le\frac{2}{\vartheta^{n/2+1}}\left(\ffint |f-(f)_{Q_{r}}-(Df)_{Q_r}(y-x)|^2dyds\right)^{1/2}  \\
&\hspace{1.5in} +\frac{2}{\vartheta^{n/2+1}}r\frac{1}{\vartheta^{n/2+1}}\left(\ffint |Df-(Df)_{Q_{r}}|^2dyds\right)^{1/2}\\
&\hspace{1in}\le \frac{2}{\vartheta^{n+2}}\left[\left(\ffint |f-(f)_{Q_{r}}-(Df)_{Q_r}(y-x)|^2dyds\right)^{1/2} \right. \\
&\hspace{1.5in} \left.+ r\left(\ffint |Df-(Df)_{Q_{r}}|^2dyds\right)^{1/2}\right].
\end{align*}
As a result, 
\begin{align}\label{backwardsDecay2}
& \fthetarint\;\;\;\left|\frac{f-(f)_{Q_{\vartheta r}}-(Df)_{Q_{\vartheta r}}(y-x)}{\vartheta r}\right|^2dyds \nonumber \\
& \hspace{.5in}\le \frac{8}{\vartheta^{2(n+3)}}\left(\ffint\;\;\left|\frac{f-(f)_{Q_{r}}-(Df)_{Q_r}(y-x)}{r}\right|^2dyds+\ffint |Df-(Df)_{Q_{r}}|^2dyds\right).\quad
\end{align}
\par 3. Let us write $\bv=(v^1,\dots, v^m)$. The conclusion follows upon substituting $f=v^i_t$ and $f=v^i_{x_j x_k}$ and summing over $i=1,\dots, m$ and $j,k=1,\dots, n$ in \eqref{backwardsDecay1}  and letting $f=v^i_{x_j}$ in \eqref{backwardsDecay2} and summing over $i=1,\dots, m$ and $j=1,\dots, n$. 
\end{proof}
\begin{cor}\label{iterCorollary}
Let $L>0$ and $\gamma\in \left(\frac{1}{2}\alpha,\alpha\right)$, select $\epsilon,\rho, \vartheta\in \left(0,\left(\frac{1}{2}\right)^{1/\alpha}\right)$ as in Lemma \ref{BlowUplemma}, and define 
$$
\epsilon_1:=\min\left\{\epsilon,\frac{\vartheta^{n/2+1}L}{8}\right\},\quad \rho_1:=\min\left\{\rho,\left(\frac{\vartheta^{2(n+3)}}{24L}\epsilon_1^2\right)^{\frac{1}{2\gamma}}\right\}.
$$
If 
\be\label{IterAssump}
\begin{cases}
Q_r:=Q_r(x,t)\subset U\times(0,T),\; r<\rho_1\\\\
|(\bv_t)_{Q_r}|, |(D\bv)_{Q_r}|, |(D^2\bv)_{Q_r}|<\frac{1}{2}L\\\\
E(x,t,r)<\epsilon_1^2,
\end{cases}
\ee
then 
\be\label{IterationCorKay1}
|(\bv_t)_{Q_{\vartheta^kr}}|, |(D\bv)_{Q_{\vartheta^kr}}|, |(D^2\bv)_{Q_{\vartheta^kr}}|\le L
\ee 
and
\be\label{IterationCorKay2}
E(x,t,\vartheta^k r)\le \frac{1}{2^k} \epsilon_1^2
\ee
for all $k\in\N$.
\end{cor}

\begin{proof}
We will prove the claim by induction on $k$, let us first study the case $k=1$. We have 
\begin{align*}
|(\bv_t)_{Q_{\vartheta r}}| &\le |(\bv_t)_{Q_{\vartheta r}}-(\bv_t)_{Q_{ r}}|+|(\bv_t)_{Q_{r}}|\\
&< \frac{1}{\vartheta^{n/2+1}}\left(\ffint |\bv_t-(\bv_t)_{Q_{r}}|^2dyds\right)^{1/2}+\frac{1}{2}L\\
&\le \frac{1}{\vartheta^{n/2+1}}E(x,t,r)^{1/2}+\frac{1}{2}L\\
&\le \frac{1}{\vartheta^{n/2+1}}\epsilon_1+\frac{1}{2}L\\
&\le L.
\end{align*}
Similarly, we can conclude that $|(D\bv)_{Q_{\vartheta r}}|\le L$ and $|(D^2\bv)_{Q_{\vartheta r}}|\le L$. This proves \eqref{IterationCorKay1} for $k=1$.

\par By Lemma \ref{BlowUplemma}, we have either
$$
E(x,t,\vartheta r)\le \frac{1}{2} \epsilon_1^2\quad \text{or}\quad E(x,t, r)\le Lr^{2\gamma}.
$$
In the case of the latter, we apply Lemma \ref{backwardsDecayLem} to get
\begin{align*}
E(x,t,\vartheta r)&\le \frac{12}{\vartheta^{2(n+3)}}E(x,t,r)\\
&\le\frac{12}{\vartheta^{2(n+3)}}Lr^{2\gamma}\\
&\le\frac{12}{\vartheta^{2(n+3)}}L\rho_1^{2\gamma}\\
&\le\frac{1}{2} \epsilon_1^2.
\end{align*}
So we have deduced \eqref{IterationCorKay1} for $k=1$. 
 
\par Now let us assume \eqref{IterationCorKay1} and \eqref{IterationCorKay2} hold for $k=1,\dots, j\ge 1$.   Generalizing our computation 
above gives 
\begin{align*}
|(\bv_t)_{Q_{\vartheta^{j+1} r}}| &\le \sum^{j+1}_{k=1}|(\bv_t)_{Q_{\vartheta^{k}r}}-(\bv_t)_{Q_{ \vartheta^{k-1} r}}|+|(\bv_t)_{Q_{r}}|\\
&\le \frac{1}{\vartheta^{n/2+1}} \sum^{j+1}_{k=1}E(x,t,\vartheta^{k-1}r)^{1/2}+\frac{1}{2}L\\
&\le \frac{1}{\vartheta^{n/2+1}} \sum^{j+1}_{k=1}\frac{\epsilon_1}{\sqrt{2}^{k-1}}+\frac{1}{2}L\\
&\le \frac{\epsilon_1}{\vartheta^{n/2+1}} \sum^{\infty}_{k=1}\left(\frac{3}{4}\right)^{k-1}+\frac{1}{2}L\\
&\le \frac{4\epsilon_1}{\vartheta^{n/2+1}} +\frac{1}{2}L\\
&\le L.
\end{align*}
Likewise, we have $|(D\bv)_{Q_{\vartheta^{j+1} r}}|\le L$ and $|(D^2\bv)_{Q_{\vartheta^{j+1} r}}|\le L$.  So we have established \eqref{IterationCorKay1} for $j+1$. 

\par By the induction hypothesis, $E(x,t,\vartheta^j r)\le \frac{1}{2^k} \epsilon_1^2\le \epsilon^2$. In particular, we have verified the hypotheses of Lemma \ref{BlowUplemma} at scale $\vartheta^j r$. 
Therefore, 
$$
E(x,t,\vartheta(\vartheta^{j}r))\le \frac{1}{2}E(x,t,\vartheta^{j}r)
$$
or 
$$
E(x,t,\vartheta^{j} r)\le L (\vartheta^jr)^{2\gamma}.
$$
In the case of the former, 
$$
E(x,t,\vartheta^{j+1}r)\le \frac{1}{2}\frac{\epsilon_1^2}{2^j}=\frac{\epsilon_1^2}{2^{j+1}}
$$
as desired. 

\par Let us now consider the latter scenario. Observe that since $\gamma>\frac{\alpha}{2}$,  
\be\label{varthetaRestrict}
0<\vartheta<\left(\frac{1}{2}\right)^{\frac{1}{\alpha}}<\left(\frac{1}{2}\right)^{\frac{1}{2\gamma}}.
\ee
In view of Lemma \ref{backwardsDecayLem} and \eqref{varthetaRestrict},
\begin{align*}
E(x,t,\vartheta^{j+1}r)&=E(x,t,\vartheta(\vartheta^{j}r))\\
&\le \frac{12}{\vartheta^{2(n+3)}}E(x,t,\vartheta^{j}r)\\
&\le \frac{12}{\vartheta^{2(n+3)}}L (\vartheta^jr)^{2\gamma}\\
&\le \frac{12}{\vartheta^{2(n+3)}}L \rho_1^{2\gamma}\cdot(\vartheta^{2\gamma})^j\\
&\le \frac{12}{\vartheta^{2(n+3)}}L \rho_1^{2\gamma}\cdot\frac{1}{2^j}\\
&\le \frac{\epsilon_1^2}{2}\cdot\frac{1}{2^j}\\
&=\frac{\epsilon_1^2}{2^{j+1}}.
\end{align*}
Therefore, in either scenario, we have verified \eqref{IterationCorKay2} for $j+1$. 
\end{proof}
The above iteration yields the following criterion for local H\"older continuity. 
\begin{cor}\label{DecayCor}
Assume $\bv$ is a weak solution of \eqref{mainPDE} in $U\times (0,T)$. Let $L>0$, $\gamma\in (\frac{1}{2}\alpha,\alpha)$ and suppose there are $(x,t)\in U\times(0,T)$ and $r>0$ as in \eqref{IterAssump}.  Then there exist $C\ge 0$, $r_1\in (0,r)$, $\mu\in (0,\gamma)$ and a neighborhood $O\subset U\times(0,T)$ of $(x,t)$ such that 
$$
E(y,s,R)\le CR^{2\mu}, \quad R\in (0,r_1),\quad (y,s)\in O. 
$$
In particular, $\bv_t$ and $D^2\bv$ are H\"older continuous in a neighborhood of $(x,t)$. 
\end{cor}
\begin{proof}
Let $R\in (0,r)$ and choose $k\in \N$ such that $\vartheta^{k+1}r<R\le \vartheta^k r$. We can derive
\begin{align*}
E(x,t,R)&\le \frac{12}{\vartheta^{2(n+3)}}E(x,t,\vartheta^k r).
\end{align*}
similarly to how we derived \eqref{backwardsDecayLemEst}.  See also Corollary 4.3 of \cite{Hynd} and Corollary 4.9 of 
\cite{HynTAM} for related estimates. In view of Corollary \ref{iterCorollary}, 
$$
E(x,t,R)\le \frac{12}{\vartheta^{2(n+3)}} \frac{\epsilon_1^2}{2^k}= \frac{24\epsilon_1^2}{\vartheta^{2(n+3)}} e^{-(k+1) \log 2}\le \frac{24\epsilon_1^2}{\vartheta^{2(n+3)}} \left(\frac{R}{r}\right)^\frac{\ln(1/2)}{\ln\vartheta}.
$$
Also note that \eqref{varthetaRestrict} implies $\mu:=\frac{1}{2}\frac{\ln(1/2)}{\ln\vartheta}< \frac{\alpha}{2}<\gamma$.

\par Observe that $(\bv_t)_{Q_\tau(y,s)}$, $(D\bv)_{Q_\tau(y,s)}$, $(D^2\bv)_{Q_\tau(y,s)}$ and $E(y,s,\tau)$ are all continuous functions of $(y,s)\in U\times (0,T)$ and $\tau>0$. Therefore, there exists an interval $(r_1, r_2)$ containing $r$ and a neighborhood $O$ of $(x,t)$ such that \eqref{IterAssump} holds for all $(y,s)\in O$ and $\rho\in(r_1,r_2)$. As a result, we may perform the same calculation above to find 
$$
E(y,s,R)\le \frac{24\epsilon_1^2}{\vartheta^{2(n+3)}} \left(\frac{R}{r_1}\right)^\frac{\ln(1/2)}{\ln\vartheta}
$$
for $(y,s)\in O$ and $R\in (0,r_1)$. Finally, the H\"older continuity of $\bv_t$ and $D^2\bv$ in a neighborhood of $(x,t)$ follows directly from 
Campanato's criterion \cite{Camp, DaPrato}. 
\end{proof}
In order to complete the proof of Theorem \ref{mainThm}, we only need to estimate the dimension of the set of points which either $\bv_t$ or $D^2\bv$ fails to be H\"older continuous.  We will express our results in terms of parabolic Hausdorff measure, so let us recall the definition.

\begin{defn}\label{ParaHausMeas}
For $G\subset \R^n\times\R$, $s\in [0,n+2]$, $\delta>0$, set 
$$
{\cal P}^s_\delta(G):=\inf\left\{\sum_{i\in \N}r^s_i: G\subset\bigcup_{i\in \N} Q_{r_i}(x_i,t_i),\; r_i\le \delta\right\}.
$$
The {\it $s$-dimensional parabolic Hausdorff measure} of $G$ is defined 
\begin{equation}
{\cal P}^s(G):=\sup_{\delta>0}{\cal P}^s_\delta(G).
\end{equation}
Moreover, the {\it parabolic Hausdorff dimension} of $G$ is the number
$$
\text{dim}_{\cal P}(G):=\inf\{s\ge 0:{\cal P}^s(G)=0 \}.
$$
\end{defn}
We note that ${\cal P}^s$ is an outer measure on $\R^n\times\R$ for each $s\in [0,n+2]$ and it is easy to check that Lebesgue outer measure on $\R^{n+1}$ is absolutely continuous with respect to ${\cal P}^{n+2}$.  General Hausdorff measure (as detailed in \cite{Evans} and \cite{Rogers}) is well studied and many important properties have been discovered. We will only make use of one fact regarding functions that are fractionally differentiable as described in the following lemma.  A close variant of the following lemma can be found in Proposition 3.3 of \cite{DuzMin}, and it can also be verified using Theorem 3 in section 2.4.3 of \cite{Evans}, Proposition 2.7 of \cite{GuistiBook}, or Lemma 4.2 in \cite{Min}, so we will not provide a proof of it.  

\begin{lem}\label{MingLemma}
Let $\kappa\in (0,1)$.  Suppose $w\in L^2_{\text{loc}}(U\times(0,T))$ 
satisfies 
\begin{equation}\label{FracSpaceW}
\int^{t_1}_{t_0}\int_{V}\int_V\frac{|w(x,t)-w(y,t)|^2}{|x-y|^{n+2\kappa}}dxdydt+\int^{t_1}_{t_0}\int^{t_1}_{t_0}\int_V\frac{|w(x,t)-w(x,s)|^2}{|t-s|^{1+\kappa}}dxdtds<\infty
\end{equation}
for each open $V\subset\subset U$ and interval $[t_0,t_1]\subset (0,T)$.   Then 
$$
\text{\normalfont{dim}}_{\cal P}\left(\left\{(x,t)\in U\times (0,T):  \limsup_{r\rightarrow 0^+}\ffintXT|w-(w)_{Q_r(x,t)}|^2dyds>0\right\}\right)\le n+2-2\kappa
$$
and
$$
\text{\normalfont{dim}}_{\cal P}\left(\left\{(x,t)\in U\times (0,T):  \limsup_{r\rightarrow 0^+}|(w)_{Q_r(x,t)}|=+\infty\right\}\right) \le n+2-2\kappa.
$$
\end{lem}
Before proceeding to a proof of Theorem \ref{mainThm}, we will need a technical lemma.

\begin{lem}
Assume $\bv$ is a weak solution of \eqref{mainPDE} on $U\times(0,T)$. There is a constant $C$ depending only on $\theta,\lambda, \Theta$, and $\Lambda$ such that 
\begin{equation}\label{MyCacci}
\iint_{Q_r(x,t)}|D\bv_t|^2dyds\le \frac{C}{r^2}\iint_{Q_{2r}(x,t)}|\bv_t-(\bv_t)_{Q_{2r}(x,t)}|^2dyds
\end{equation}
whenever $Q_{2r}(x,t)\subset U\times(0,T)$. 
\end{lem}

\begin{proof}
Fix $c\in \R^m$ and $\phi\in C_c^\infty(U\times(0,T))$. As we derived \eqref{SecondIdentity}, we find
\begin{align}
&\frac{d}{dt}\int_U \phi (D\psi(\bv_t)\cdot(\bv_t-c)-\psi(\bv_t)+\psi(c) )dx+\int_U\phi D^2F(D\bv)(D\bv_t, D\bv_t) dx=\\
&\hspace{.5in}\int_U\left((D\psi(\bv_t)\cdot(\bv_t-c)-\psi(\bv_t)+\psi(c) )\phi_t -(\bv_t-c)\cdot D^2F(D\bv)D\bv_tD\phi \right)dx
\end{align}
for almost every $t\in (0,T)$. Selecting $\phi=\eta^2$ for $\eta\ge 0$ gives the inequality
\be\label{MyCacciOnTheWay}
\max_{0\le t\le T}\int_U\eta^2|\bv_t-c|^2 dx + \int^T_0\int_U \eta^2|D\bv_t|^2dxdt 
\le C_1\int^T_0\int_U\left(\eta|\eta_t|+|D\eta|^2\right) |\bv_t-c|^2dxdt.
\ee
Here $C_1$ only depends on $\theta,\lambda, \Theta$ and $\Lambda$.

\par In order to conclude, we pick $c=(\bv_t)_{Q_{2r}(x,t)}$ and
choose $\eta$ to satisfy $\eta(y,\tau) = \eta_0(y)\cdot \eta_1(\tau)$, where
$$
\begin{cases}
\eta_0\in C^\infty_c(\R^n)\\
0\le \eta_0\le 1\\
\eta_0\equiv 1\;\;\text{in}\;\; B_r(x)\\
\eta_0\equiv 0\;\;\text{in}\;\; \R^n\setminus B_{2r}(x)\\
|D\eta_0|\le 2/r
\end{cases}\quad\text{and}\quad \quad
\begin{cases}
\eta_1\in C^\infty_c(\R)\\
0\le \eta_1\le 1\\
\eta_1\equiv 1\;\;\text{in}\;\; (t-r^2/2,t+r^2/2)\\
\eta_1\equiv 0\;\;\text{in}\;\; \R\setminus (t-2r^2,t+2r^2)\\
|\partial_\tau\eta_1|\le 2/r^2.
\end{cases}
$$
We only need to substitute this choice in \eqref{MyCacciOnTheWay} to conclude \eqref{MyCacci}. 
\end{proof}

\begin{proof}[Proof of Theorem \ref{mainThm}]  
Our goal is to show ${\cal P}^{n+2-2\beta}(U\times(0,T)\setminus{\cal O})=0$ for some $\beta\in (0,1)$ where
$$
{\cal O}:=\{(x,t)\in U\times(0,T): D^2\bv\;\text{and}\; \bv_t \;\text{are H\"older continuous in a neighborhood of $(x,t)$}\}.
$$
To this end, we choose $\beta, \epsilon>0$ to satisfy 
\begin{equation}\label{BetaInterval3}
\beta+\epsilon<\min\left\{\frac{\alpha}{2},\frac{1}{2}-\frac{1}{p}\right\}.
\end{equation}
Here $p>2$ is from Corollaries \ref{PeeCorollary} and \ref{TimeDiffDVCor} and $\alpha$ is a H\"older exponent for $D^2F$ that 
we considered in Corollary \ref{SecondDiffDVCor}.  By Corollary \ref{DecayCor}, 
$$
U\times(0,T)\setminus{\cal O}\subset G_1 \cup G_2\cup G_3\cup G_4
$$
where
$$
\begin{cases}
G_1:=\left\{(x,t)\in U\times (0,T):  \limsup_{r\rightarrow 0^+}E(x,t,r)>0\right\}\\\\
G_2:=\left\{(x,t)\in U\times (0,T):  \limsup_{r\rightarrow 0^+}|(\bv_t)_{Q_r(x,t)}|=+\infty\right\}\\\\
G_3:=\left\{(x,t)\in U\times (0,T):  \limsup_{r\rightarrow 0^+}|(D\bv)_{Q_r(x,t)}|=+\infty\right\}\\\\
G_4:=\left\{(x,t)\in U\times (0,T):  \limsup_{r\rightarrow 0^+}|(D^2\bv)_{Q_r(x,t)}|=+\infty\right\}.
\end{cases}
$$
It suffices to show ${\cal P}^{n+2-2\beta}(G_i)=0$ for $i=1,2,3,4.$

\par Let us recall Poincar\'e's inequality on the cylinder $Q_r(x,t)\subset U\times(0,T)$
$$
\iint_{Q_r}|\bw - (\bw)_{Q_r}|^2dyds \le C_0\left\{r^4\iint_{Q_r}|\bw_t|^2dyds + r^2\iint_{Q_r}|D\bw|^2dyds\right\}
$$
for $\bw\in H^1(U\times(0,T);\R^m)$. Here $C_0$ is a constant independent of $\bw$. Choosing 
$$
\bw=\bv_{x_i}-(\bv_{x_i})_{Q_r} -(D\bv_{x_i})_{Q_r}(y-x),
$$ 
summing over $i=1,\dots, m$ and dividing by $r^2$ gives
{\small 
\be\label{SimpleUpperBoundE}
E(x,t,r)\le (1+C_0)\left\{ 
\ffint|\bv_t - (\bv_t)_{Q_r}|^2 dyds+ r^2\ffint\;\; \left|D\bv_t\right|^2dyds + \ffint\;\; \left|D^2\bv-(D^2\bv)_{Q_r}\right|^2dyds\right\}.
\ee}

\par Using \eqref{MyCacci}, we can take the limit superior of both sides of \eqref{SimpleUpperBoundE} to get 
\begin{align*}
\limsup_{r\rightarrow 0^+}E(x,t,r)&\le (C_0+1)(C+1)\limsup_{r\rightarrow 0^+}\ffint|\bv_t - (\bv_t)_{Q_r}|^2 dyds\\
&\hspace{1in}+(C_0+1)\limsup_{r\rightarrow 0^+}\ffint\;\; \left|D^2\bv-(D^2\bv)_{Q_r}\right|^2dyds
\end{align*}
for any $(x,t)\in U\times (0,T)$.  Therefore, $G_1\subset G_5\cup G_6$ where
$$
G_5:=\left\{(x,t)\in U\times (0,T):  \limsup_{r\rightarrow 0^+} \ffintXT|\bv_t-(\bv_t)_{Q_r(x,t)}|^2dyds>0\right\}
$$
and 
$$
G_6:=\left\{(x,t)\in U\times (0,T):  \limsup_{r\rightarrow 0^+} \ffintXT|D^2\bv-(D^2\bv)_{Q_r(x,t)}|^2dyds>0\right\}.
$$

\par Since $D\bv_t\in L^2_{\text{loc}}(U\times(0,T);\Mmn)$ and $\bv_t$ satisfies 
\eqref{FractionVtDeriv2}, the components of $\bv_t$ satisfies \eqref{FracSpaceW} for $\kappa=\beta +\epsilon$.   Here we are using the 
fact that $H^1_{\text{loc}}(U)\subset H^\sigma_{\text{loc}}(U)$ $(0<\sigma<1)$ (Proposition 2.2 of \cite{DiHitch}). Lemma 
\ref{MingLemma} then implies  
$$
\text{dim}_{\cal P}(G_5)\le n+2-2(\beta+\epsilon)\quad \text{and}\quad \text{dim}_{\cal P}(G_2)\le n+2-2(\beta+\epsilon).
$$
It follows that ${\cal P}^{n+2-2\beta}(G_5)={\cal P}^{n+2-2\beta}(G_2)=0$.  Likewise, we can make use of the integrability 
$D^3\bv\in L^2_{\text{loc}}(U\times(0,T);{\cal S}^3(\R^n;\R^m))$ and fractional time differentiability \eqref{FractionD2VDeriv} to show 
that $v^i_{x_jx_k}$ satisfies \eqref{FracSpaceW} for $\kappa=\beta +\epsilon$. Using Lemma 
\ref{MingLemma}, we have ${\cal P}^{n+2-2\beta}(G_6)={\cal P}^{n+2-2\beta}(G_4)=0$. Hence, ${\cal P}^{n+2-2\beta}(G_1)=0$, as 
well.  The conclusion ${\cal P}^{n+2-2\beta}(G_3)=0$ follows similarly as $v^i_{x_j}$ satisfies \eqref{FracSpaceW} for every $\kappa\in (0,1)$ and $i=1,\dots, m$, $j=1,\dots, n$.
\end{proof}

%%%%%%%%%%%%%%%%%%%%%%%%%%%%%%%%%%%%%%%%%%%%%%%%%%%%%%%%%
\appendix
\section{The Dirichlet problem}\label{DirichletAppendix}
In this appendix, we consider the following initial value problem: for a given $\bg:  U\rightarrow \R^m$, 
find a solution $\bv: U\times [0,T)\rightarrow \R^m$ of 
\begin{equation}\label{mainIVP}
\begin{cases}
D\psi(\bv_t)=\Div DF(D\bv), \quad & \text{in}\;U\times(0,T)\\
\hspace{.43in}\bv=0, \quad &  \text{on}\;\partial U\times[0,T)\\
\hspace{.43in}\bv=\bg, \quad &  \text{on}\;U\times\{0\}.
\end{cases}
\end{equation} 
It has been shown that this initial value problem has a solution, which is 
known to satisfy a global estimate of the type \eqref{EnergyBound1}. Our goal here is to show that there exists a weak solution that additionally satisfies inequality \eqref{EnergyBound2}. Applying Theorem \ref{mainThm}, we will also be able to conclude that this weak solution is in 
fact partially regular.

For any smooth solution $\bv$, we have
\begin{equation}\label{Aidentity1}
\frac{d}{dt}\int_U F(D\bv)dx=-\int_UD\psi(\bv_t)\cdot \bv_t dx.
\end{equation}
It then follows
\begin{equation}\label{DirIdenitity1}
\int^T_0\int_U|\bv_t|^2dxds+\max_{0\le t\le T}\int_U|D\bv(x,t)|^2dx\le C\int_U|D\bg|^2dx.
\end{equation}
The constant $C$ only depends on $\theta, \Theta,\lambda$ and $\Lambda$. 

\par We also have 
\be \label{Aidentity2}
\frac{d}{dt}\int_U \psi^*(D\psi(\bv_t))dx = - \int_U D^2F(D\bv)\left(D\bv_t, D\bv_t\right)dx.
\ee
We can multiply this identity with $f\in C^\infty_c(0,T)$ that satisfies $0\le f\le 1$,  $|f'|\le 2/d$ and $f(t)=1$ for $t\in [d,T-d]$ for some $d\in (0,T/2)$; integrating
the resulting equality over $(0,T)$, we find there is constant $C$ such that 
\be\label{DirIdenitity2}
\max_{d\le t\le T-d}\int_U|\bv_t(x,t)|^2dx +\int^{T-d}_{d}\int_U|D\bv_t|^2dxds\le \frac{C}{d}\int_U|D\bg|^2dx.
\ee
Here $C$ only depends on $\theta, \Theta,\lambda,$ and $\Lambda$.  This bound along with \eqref{DirIdenitity1} gives us an idea of what type of integrability can be expected from a weak solution. In 
particular, we have the following definition. 

\begin{defn}\label{DirichletWeakSoln} Suppose $\bg\in H^1(U;\R^m)$. A {\it weak solution} $\bv$ of  \eqref{mainIVP} satisfies $(i)$  
\be\label{NaturalIVPSpaceZero}
\bv\in L^\infty((0,T); H^{1}_{0}(U;\R^m)),
\ee
\begin{equation}\label{NaturalIVPSpace}
\bv_t\in L^\infty_{\text{loc}}((0,T); L^2(U;\R^m))\cap L^2(U\times(0,T);\R^m),
\end{equation}
and 
\begin{equation}\label{NaturalIVPSpace2}
D\bv_t\in L^2_{\text{loc}}(U\times (0,T); \Mmn),
\end{equation}
$(ii)$ the weak solution condition \eqref{WeakSolnCond} and $(iii)$
\be \label{initialCondfov}
\bv(\cdot, 0)=\bg.
\ee
\end{defn}
\begin{rem}
The integrability \eqref{NaturalIVPSpaceZero} and \eqref{NaturalIVPSpace} imply that $\bv: [0,T)\rightarrow L^2(U;\R^m)$ is continuous. Thus, 
we have no problem setting \eqref{initialCondfov}. 
\end{rem}

\par We will now provide an approach to verifying the existence of a weak solution as defined above.  To this end, we will employ the {\it implicit time scheme}:  $N\in\N$, $\tau:=T/N$, $\bv^0=\bg$
\begin{equation}\label{IFT}
\begin{cases}
D\psi\left(\displaystyle\frac{\bv^k - \bv^{k-1}}{\tau}\right)=\Div DF(D\bv^k), \quad &\;\text{in}\; U\\
\hspace{1.1in}\bv^k=0,\quad & \text{on}\; \partial U
\end{cases}
\end{equation}
for  $ k=1,\dots, N$.  Here \eqref{IFT} holds in the weak sense: for $k=1,\dots, N$,
\begin{equation}\label{WeakIFT}
\int_U D\psi\left(\frac{\bv^k - \bv^{k-1}}{\tau}\right)\cdot  \bw dx + \int_U DF(D\bv^k)\cdot D\bw dx=0 
\end{equation}
for each $\bw \in H^1_0(U;\R^m)$.  We now present two fundamental identities for this discrete scheme that are inspired by 
\eqref{Aidentity1} and \eqref{Aidentity2}.

\begin{prop}
Let $\{\bv^1,\dots, \bv^N\}$ be a solution sequence of \eqref{IFT}.  Then 
\begin{equation}\label{DiscreteIdentity1}
\int_U D\psi\left(\frac{\bv^k - \bv^{k-1}}{\tau}\right)\cdot  (\bv^k - \bv^{k-1})dx + \int_U F(D\bv^k)dx\le \int_U F(D\bv^{k-1})dx
\end{equation}
for $k=1,\dots, N$ and
\begin{align}\label{DiscreteIdentity2}
&\int_U \left(DF(D\bv^k) - DF(D\bv^{k-1})\right)\cdot (D\bv^k - D\bv^{k-1})dx \nonumber \\
&+ \tau\int_U\psi^*\left(D\psi\left(\frac{\bv^k-\bv^{k-1}}{\tau}\right)\right)dx\le\tau\int_U\psi^*\left(D\psi\left(\frac{\bv^{k-1}-\bv^{k-2}}{\tau}\right)\right)dx
\end{align}
$k=2,\dots, N$.
\end{prop}
\begin{proof}
Choosing $\bw=\bv^k-\bv^{k-1}$ in \eqref{WeakIFT} and using the convexity of $F$ gives \eqref{DiscreteIdentity1}. Let us now focus on 
\eqref{DiscreteIdentity2}.  In view of \eqref{WeakIFT},
\begin{align*}
\int_U \left(DF(D\bv^k) - DF(D\bv^{k-1}\right))\cdot (D\bv^k - D\bv^{k-1})dx &= \int_U  DF(D\bv^k) \cdot D(\bv^k - \bv^{k-1})dx\\
&\;\; -\int_U  DF(D\bv^{k-1}) \cdot D(\bv^k - \bv^{k-1})dx\\
 &= - \int_U D\psi\left(\frac{\bv^k - \bv^{k-1}}{\tau}\right)\cdot (\bv^k - \bv^{k-1})dx\\
& \;\; \int_U D\psi\left(\frac{\bv^{k-1} - \bv^{k-2}}{\tau}\right)\cdot (\bv^k - \bv^{k-1})dx.
\end{align*}
The claim follows once we notice 
\begin{align*}
&\int_U \left[D\psi\left(\frac{\bv^{k-1} - \bv^{k-2}}{\tau}\right)-D\psi\left(\frac{\bv^k - \bv^{k-1}}{\tau}\right)\right]\cdot (\bv^k - \bv^{k-1})dx \\
& =\tau\int_U \left[D\psi\left(\frac{\bv^{k-1} - \bv^{k-2}}{\tau}\right)-D\psi\left(\frac{\bv^k - \bv^{k-1}}{\tau}\right)\right]\cdot D\psi^*\left(D\psi\left(\frac{\bv^k - \bv^{k-1}}{\tau}\right)\right)dx\\
&\le \tau\int_U\left[\psi^*\left(D\psi\left(\frac{\bv^{k-1}-\bv^{k-2}}{\tau}\right)\right)-\psi^*\left(D\psi\left(\frac{\bv^k-\bv^{k-1}}{\tau}\right)\right)\right]dx. 
\end{align*}
The inequality here follows by the convexity of $\psi^*$. 
\end{proof}
Let us denote $\tau_k=k\tau $ for $k=0,\dots, N$. For a given $d\in (0,T/2)$, we  will also use the notation $d\le \tau_k\le T-d$ to denote the collection of natural numbers
$$
\{k\in \N: d\le \tau_k\le T-d \}.
$$
Below, we present two estimates that are discrete analogs of \eqref{DirIdenitity1} and \eqref{DirIdenitity2}. 
\begin{cor}
Let $\{\bv^1,\dots, \bv^N\}$ be a solution sequence of \eqref{IFT} and $d\in (0,T/2)$. There is a constant $C$ depending only on $\theta,\Theta,\lambda, \Lambda$ such that 
 \begin{equation}\label{DiscreteBound1}
\sum^N_{k=1}\int_U\frac{|\bv^k - \bv^{k-1}|^2}{\tau}dx + \max_{1\le k\le N}\int_U |D\bv^k|^2dx\le C\int_U |D\bg|^2dx
\end{equation}
and 
\begin{equation}\label{DiscreteBound2}
\sum_{d\le \tau_k\le T-d}\int_U\frac{|D\bv^k - D\bv^{k-1}|^2}{\tau}dx + \max_{d\le \tau_{k}\le T-d}\int_U
\left|\frac{\bv^{k-1} - \bv^{k-2}}{\tau}\right|^2dx\le \frac{C}{d}\int_U |D\bg|^2dx
\end{equation} 
for all $N$ sufficiently large. 
\end{cor}
\begin{proof}
Summing \eqref{DiscreteIdentity1} over $k=1,\dots, j\le N$ gives
\be
\sum^j_{k=1}\int_U D\psi\left(\frac{\bv^k - \bv^{k-1}}{\tau}\right)\cdot  (\bv^k - \bv^{k-1})dx + \max_{1\le k\le j}\int_U F(D\bv^k)dx\le \int_U F(D\bg)dx. 
\ee
Our assumptions on the convexity of $\psi$ and $F$ now immediately imply \eqref{DiscreteBound1}.

\par Let us now consider \eqref{DiscreteBound2}. Choose $f\in C^\infty_c(0,T)$ with $0\le f\le 1$, $f(t)= 1$ for $t\in [d,T-d]$  and $|f'|\le 2/d$. Multiplying
\eqref{DiscreteIdentity2} by $f(\tau_k)$ and summing over $k=2,\dots, j\le N$ gives
\begin{align*}
&\sum^j_{k=2}f(\tau_k)\int_U \left(DF(D\bv^k) - DF(D\bv^{k-1}\right)\cdot (D\bv^k - D\bv^{k-1})dx \\
&\le \tau\sum^j_{k=2}f(\tau_k)\int_U\left[\psi^*\left(D\psi\left(\frac{\bv^{k-1}-\bv^{k-2}}{\tau}\right)\right)-\psi^*\left(D\psi\left(\frac{\bv^k-\bv^{k-1}}{\tau}\right)\right)\right]dx.
\end{align*}
Now let $N$ be so large that $f(t)=0$ for $t\in [0,\tau_2]=[0,2T/N]$. Summing by parts, we have 
\begin{align*}
&\tau\sum^j_{k=2}f(\tau_k)\int_U\left[\psi^*\left(D\psi\left(\frac{\bv^{k-1}-\bv^{k-2}}{\tau}\right)\right)-\psi^*\left(D\psi\left(\frac{\bv^k-\bv^{k-1}}{\tau}\right)\right)\right]dx
\\
&= -\left[\tau f(\tau_j)\int_U\psi^*\left(D\psi\left(\frac{\bv^{j-1}-\bv^{j-2}}{\tau}\right)\right)dx-\tau f(\tau_2)\int_U\psi^*\left(D\psi\left(\frac{\bv^{1}-\bv^{0}}{\tau}\right)\right)dx\right]
\\
&\quad + \tau\sum^j_{k=2}(f(\tau_k)-f(\tau_{k-1}))\int_U\psi^*\left(D\psi\left(\frac{\bv^{k-1}-\bv^{k-2}}{\tau}\right)\right)dx\\
&=\tau\sum^j_{k=2}(f(\tau_k)-f(\tau_{k-1}))\int_U\psi^*\left(D\psi\left(\frac{\bv^{k-1}-\bv^{k-2}}{\tau}\right)\right)dx\\
&\quad\quad -\tau f(\tau_j)\int_U\psi^*\left(D\psi\left(\frac{\bv^{j-1}-\bv^{j-2}}{\tau}\right)\right)dx.
\end{align*}
\par Therefore,
\begin{align*}
&\sum^j_{k=2}f(\tau_k)\int_U \left(DF(D\bv^k) - DF(D\bv^{k-1}\right)\cdot \frac{D\bv^k - D\bv^{k-1}}{\tau}dx+\\
&  f(\tau_j)\int_U\psi^*\left(D\psi\left(\frac{\bv^{j-1}-\bv^{j-2}}{\tau}\right)\right)dx\le \sum^j_{k=2}(f(\tau_k)-f(\tau_{k-1}))\int_U\psi^*\left(D\psi\left(\frac{\bv^{k-1}-\bv^{k-2}}{\tau}\right)\right)dx
\end{align*}
$j=2,\dots, N$.  From our convexity assumptions on $\psi$ and $F$ and the assumption that $|f'|\le 2/d$.
\begin{align*}
&\lambda\sum^N_{k=1}f(\tau_k)\int_U\frac{|D\bv^k- D\bv^{k-1}|^2}{\tau}dx+\frac{\theta}{2}\max_{1\le k\le N}f(\tau_k)\int_U\left|\frac{\bv^{k-1}-\bv^{k-2}}{\tau}\right|^2dx\\
& \hspace{1in} \le\frac{\Theta}{d} \sum^N_{k=1}\int_U\frac{\left|\bv^{k}-\bv^{k-1}\right|^2}{\tau}dx.
\end{align*}
Since $f(\tau_k)=1$ whenever $d\le\tau_k\le T-d$, we are able to conclude \eqref{DiscreteBound2} provided
$\sum^N_{k=1}\int_U\left|\bv^{k}-\bv^{k-1}\right|^2/\tau dx\le C\int_U|D\bg|^2dx$, which holds from \eqref{DiscreteBound1}. 
\end{proof}

\par  Let us now see how the estimates \eqref{DiscreteBound1} and \eqref{DiscreteBound2} can be used to show the existence of 
a weak solution of \eqref{mainIVP}. First, we define
\begin{equation}\label{StepApprox}
\begin{cases}
\bv_N(x, t)=\bg(x), \hspace{.23in} t=0\\
\hspace{.6in}                 =\bv^{k}(x), \quad t\in (\tau_{k-1}, \tau_k]
\end{cases}
\end{equation}
and 
$$
\bu_N(x, t)  = \bv^{k-1}(x) +\frac{t-\tau_{k-1}}{\tau}(\bv^k(x)-\bv^{k-1}(x)), \quad t\in [\tau_{k-1}, \tau_k].
$$
Observe by \eqref{WeakIFT}, we have that for all $t\in [0,T]\setminus\{\tau_0,\tau_1,\dots,\tau_N\}$ and $\bw\in H^1_0(U;\R^m)$
$$
\int_U D\psi\left(\partial_t \bu_N(x,t)\right)\cdot  \bw(x)dx + \int_U DF(D\bv_N(x,t))\cdot D\bw(x)dx=0.
$$

\par By inequality \eqref{DiscreteBound1} and a few routine manipulations, we also have 
$$
\sup_{N\in \N}\left\{\int^T_0\int_U |\partial_t\bu_N|^2dxdt+\max_{0\le t\le T}\int_U(|D\bv_N(x,t)|^2+|D\bu_N(x,t)|^2)dx\right\}<\infty.
$$ 
Using this uniform bound and ideas given in the proof of Proposition \ref{CompactnessProp}, it can be shown (see for example \cite{Arai, Colli, Mielke, Visintin}) that there is a mapping $\bv\in C([0,T); L^2(U;\R^m))$ and a sequence $N_j\rightarrow \infty$ for which
$$
\bu_{N_j}, \bv_{N_j}\rightarrow \bv \quad \text{in}\quad L^2([0,T]; H^1_0(U;\R^m))\cap C([0,T); L^2(U;\R^m))
$$
and 
$$
\partial_t\bu_{N_j}\rightarrow \partial_t\bv  \quad \text{in}\quad L^2(U\times [0,T];\R^m).
$$
Moreover, $\bv$ satisfies the weak solution condition  \eqref{WeakSolnCond}.

\par All that remains to be verified is that $\bv$ satisfies the integrability \eqref{NaturalIVPSpace} and \eqref{NaturalIVPSpace2}.  Fortunately, we have \eqref{DiscreteBound2} at our disposal. For each $d\in (0,T/2)$, this estimate implies that the sequence 
$$
\max_{d\le t\le T-d}\int_U|\partial_t\bu_N(x,t)|^2dx+\int^{T-d}_d\int_U  |\partial_tD\bu_N(x,t)|^2dxdt
$$
is bounded for all sufficiently large $N\in \N$. From this boundedness property, we immediately have that that $\bv$ satisfies \eqref{NaturalIVPSpace}. We also have upon passing to a further subsequence if necessary that $\partial_tD\bu_{N_j}\rightharpoonup \xi$ in $L^2(U\times [d,T-d];\Mmn)$.  
Note for $\Psi\in C^\infty_c(U\times (d,T-d); \R^m)$
\begin{align*}
\int^{T-d}_d\int_U \xi\Psi dxdt & = \lim_{j\rightarrow\infty}\int^{T-d}_d\int_U\left(\partial_tD\bu_{N_j}\right)\Psi dxdt \\
&=-\lim_{j\rightarrow\infty}\int^{T-d}_d\int_U D\bu_{N_j}\Psi_t dxdt \\
&=  -\int^{T-d}_d\int_U D\bv \Psi_tdxdt.
\end{align*}
As a result, $D\bv_t|_{U\times [d,T-d]}=\xi$. It follows that $D\bv_t\in L^2_\text{loc}(U\times (0,T); \Mmn)$. Therefore, we conclude the existence of a weak solution of \eqref{mainIVP} as defined in Definition \ref{DirichletWeakSoln}. 
\begin{prop}
There exists a weak solution of \eqref{mainIVP}.
\end{prop}

% References

\end{document}